\documentclass{amsart}
\usepackage{amssymb} 
\usepackage{graphicx} 
\usepackage{dsfont} 
\usepackage{tikz}
\usepackage{quiver}
\usepackage{hyperref}
\usepackage[shortlabels]{enumitem} 

\hypersetup{
  colorlinks   = true, 
  urlcolor     = cyan, 
  linkcolor    = blue, 
  citecolor   = black 
}

\usepackage[backend=biber,style=alphabetic,citestyle=alphabetic,url=false,isbn=false,maxnames=5,maxalphanames=5]{biblatex}
\makeatletter
\g@addto@macro\th@plain{\thm@headpunct{}}
\g@addto@macro\th@definition{\thm@headpunct{}}
\g@addto@macro\th@remark{\thm@headpunct{}}
\makeatother

\makeatletter
\newtheorem*{rep@theorem}{\rep@title}
\newcommand{\newreptheorem}[2]{%
\newenvironment{rep#1}[1]{%
 \def\rep@title{#2 \ref{##1}}%
 \begin{rep@theorem}}%
 {\end{rep@theorem}}}
\makeatother

\newtheorem{theorem}{Theorem}[section]
\newreptheorem{theorem}{Theorem}
\newtheorem{proposition}[theorem]{Proposition}
\newreptheorem{proposition}{Proposition}

\newtheorem{corollary}[theorem]{Corollary}
\newreptheorem{corollary}{Corollary}
\newtheorem{lemma}[theorem]{Lemma}
\newtheorem{conjecture}[theorem]{Conjecture}
\theoremstyle{definition}
\newtheorem{example}[theorem]{Example}
\newtheorem{definition}[theorem]{Definition}

\newtheorem{remark}[theorem]{Remark}
\newtheorem{question}[theorem]{Question}

\newcommand{\1}{\mathds 1}
\newcommand{\B}{\mathcal B}
\newcommand{\C}{\mathcal C}
\newcommand{\D}{\mathcal D}
\newcommand{\E}{\mathcal E}
\newcommand{\Z}{\mathcal Z}
\newcommand{\CC}{\mathbb C}
\newcommand{\DD}{\mathbb D}
\newcommand{\RR}{\mathbb R}
\newcommand{\HH}{\mathbb H}
\newcommand{\KK}{\mathbb K}
\newcommand{\ZZ}{\mathbb Z}
\newcommand{\Mod}{\mathrm{Mod}}
\newcommand{\Hom}{\mathrm{Hom}}
\newcommand{\End}{\mathrm{End}}
\newcommand{\Rep}{\mathrm{Rep}}
\newcommand{\Aut}{\mathrm{Aut}}

\newcommand{\Pic}{\mathcal Pic}
\newcommand{\BrPic}{\mathcal Br\mathcal Pic}
\newcommand{\Witt}{\mathcal Witt}
\newcommand{\Mext}{\mathcal Mext}
\renewcommand{\Vec}{\mathrm{Vec}}
\newcommand{\sVec}{\mathbf{s}\hspace{-1pt}\mathrm{Vec}}
\newcommand{\FPdim}{\mathrm{FPdim}}
\newcommand{\id}{\mathrm{id}}

\newcommand{\Tri}{\mathbb{T}\mathrm{ri}}
\newcommand{\TC}{\mathrm{TC}}
\newcommand{\TF}{\mathrm{3F}}
\newcommand{\DS}{\mathrm{DS}}
\newcommand{\Sem}{\mathcal Sem}
\newcommand{\Q}{\mathcal Q}
\newcommand{\CF}{\mathrm{CF}}
\newcommand{\Gr}{\mathrm{Gr}}

\usepackage{amsaddr}
\usepackage{orcidlink}

\author[S. Sanford]{Sean Sanford\, \orcidlink{0000-0002-2439-3764}}
\address{School of Mathematics, The University of Edinburgh}
\email{ssanford@ed.ac.uk}
\title{On the Mext groups of \texorpdfstring{$\sVec_{\RR}$}{sVec_R} and \texorpdfstring{$\sVec_{\HH}$}{sVec_H}}
\date{September 2026}

\begin{document}

\begin{abstract}
    We compute the groups of minimal nondegenerate extensions of the real symmetric fusion categories $\sVec_\RR$ and $\sVec_\HH$.
    We find that they are both isomorphic to the Klein-four group $(\ZZ/2\ZZ)^2$.
    Along the way, we classify all nondegenerately braided fusion categories over $\RR$ that have Frobenius-Perron dimension 4, and we determine exactly which of these categories is a Drinfeld center.
    We end the paper by discussing a homotopy-theoretic conjecture that organizes all these structures.
\end{abstract}

\maketitle

The classification of minimal nondegenerate extensions (MNEs) of braided fusion categories has attracted the attention of many researchers in both mathematics and condensed matter physics \cite{
MR1990929,
MR2200691,
MR3641612,
MR3613518,
MR3775361,
venegasramírez2019minimalmodularextensionssupertannakian,
MR4281262,
MR4496388,
MR4432444,
MR4504933,
MR4560997,
MR4654609,
MR4920627,
MR4850473,
MR4971769,
MR5089380,
johnsonfreyd2026structurewittgroupsminimal}.
For a symmetry protected topological phase of matter, hidden degrees of freedom that are protected by the symmetry can often be detected by gauging the symmetry.
According to Lan, Kong, and Wen \cite{MR3613518}, minimal nondegenerate extensions are a mathematical way of performing this gauging process for (2+1)D theories.

Let us briefly describe what an MNE is.
An object in a braided fusion category is said to be transparent if it braids trivially with all other objects, and a braiding is said to be nondegenerate if $\1$ is the only transparent simple object.
A nondegenerate extension of a braided category $\mathcal E$ is a braided embedding of $\mathcal E$ into some category $\mathcal C$, whose braiding is nondegenerate.
Such an extension is minimal if the only objects in $\mathcal C$ that are transparent to all of $\mathcal E$ were already in $\mathcal E$.

Nondegeneracy of a braiding is analogous to nondegeneracy of a bilinear form, and the theory of braided fusion categories can be seen as a generalization of the theory of pre-metric groups (\emph{i.e.} groups equipped with a quadratic form).
Nikshych emphasizes this analogy in the introduction to \cite{MR4504933}, where he interprets it as telling us that the theory of braided fusion categories behaves like a kind of `categorified linear algebra'.
This generalization follows from a theorem of Eilenberg and Mac Lane \cite{MR65162}, which was later translated into the language of braided monoidal categories in \cite{MR1250465}, and interpreted in the modern language of braided fusion categories in \cite{MR2609644}.

The collection of MNEs for the category $\sVec_\CC$ of super-vector spaces, denoted $\Mext(\sVec_\CC)$, was computed by Kitaev in \cite{MR2200691}, where he found a $\ZZ/16\ZZ$ group structure, now known as Kitaev's 16-fold way.
Lan, Kong, and Wen later generalized this in \cite{MR3613518} by proving that for any symmetric fusion category $\E$, there is a natural group structure on $\Mext(\E)$.
This group structure can be understood in terms of Witt groups.

The operation of direct sum endows the collection of all finite metric groups with the structure of an abelian monoid.
The quotient of this monoid by hyperbolic metric groups becomes an abelian group known as the Witt group of metric groups, denoted in the literature by $\Witt_{ptd}$.
In \cite{MR3039775}, Davydov, Müger, Nikshych, and Ostrik introduced an analogous construction for nondegenerately braided fusion categories, denoted $\Witt(\Vec_\CC)$, and showed that this contains $\Witt_{ptd}$ as a subgroup.
Davydov, Nikshych, and Ostrik later showed in \cite{MR3022755} that there is a super version $\Witt(\sVec_\CC)$ of the Witt group, and that $\Mext(\sVec_\CC)$ can be canonically identified with the kernel of the `superification functor'
\[(-)\boxtimes\sVec_\CC:\Witt(\Vec_\CC)\to\Witt(\sVec_\CC)\,.\]
A more general version of this statement was proven in \cite{MR3613518}, where they established that there is always a surjective map
\[U_{\E}:\Mext(\E)\twoheadrightarrow\ker\Big((-)\boxtimes\E:\Witt(\Vec_\CC)\to\Witt(\E)\Big)\,,\]
for any symmetric fusion category $\E$ over $\CC$.

This functor and its kernel can be interpreted in the language of higher Galois theory, originally outlined in \cite{MR3623677}.
In this paper, Jonson-Freyd invites us to think of $\Vec_\RR$, $\Vec_\CC$, and $\sVec_\CC$ as categorified fields, and he offers an interpretation of $\sVec_\CC$ as the algebraic closure of $\Vec_\CC$.
From this perspective, the Witt groups are higher analogues of Brauer groups, and $\Mext$ corresponds to an analogue of the relative Brauer group of an extension.

In addition to these familiar categorified fields, Jonson-Freyd observes that there is one additional symmetric fusion category that lies between $\Vec_\RR$ and $\sVec_\CC$: $\sVec_\HH$.
Its existence can be seen from the Galois correspondence, and the fact that there is a `subgroup' $\ZZ/2\ZZ\hookrightarrow\ZZ/2\ZZ\times B\ZZ/2\ZZ$, corresponding to the product map $(\id,\mathrm{Sq}^1)$.
The Hasse diagram of categorified field extensions is shown below.
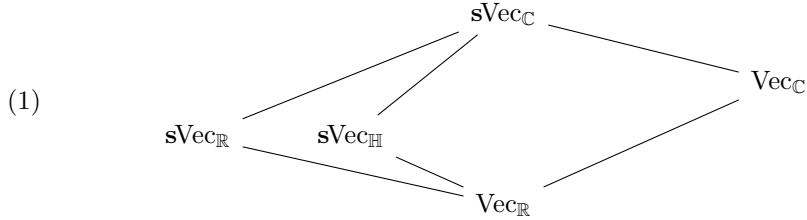
\begin{equation}\label{eqn:Hasse diagram}
    \begin{tikzcd}[row sep=5,ampersand replacement=\&]
    	\&\& {\mathbf{s}\kern-1pt\mathrm{Vec}_{\mathbb C}} \&\&\& \\
    	\\
    	\&\&\&\&\& {\mathrm{Vec}_{\mathbb C}} \\
    	{\mathbf{s}\kern-1pt\mathrm{Vec}_{\mathbb R}} \& {\mathbf{s}\kern-1pt\mathrm{Vec}_{\mathbb H}} \\
    	\\
    	\&\& {\mathrm{Vec}_{\mathbb R}}
    	\arrow[no head, from=3-6, to=1-3]
    	\arrow[no head, from=4-1, to=1-3]
    	\arrow[no head, from=4-2, to=1-3]
    	\arrow[no head, from=6-3, to=3-6]
    	\arrow[no head, from=6-3, to=4-1]
    	\arrow[no head, from=6-3, to=4-2]
    \end{tikzcd}
\end{equation}

Just as $\Mext(\sVec_\CC)$ was used to aid in the analysis of $\Witt(\Vec_\CC)$ and $\Witt(\sVec_\CC)$ in \cite{MR3022755}, it is reasonable to expect that knowing $\Mext(\sVec_\RR)$ and $\Mext(\sVec_\HH)$ will be helpful when trying to compute $\Witt(\Vec_\RR)$, $\Witt(\sVec_\RR)$, and $\Witt(\sVec_\HH)$.
In this paper, we will compute both $\Mext(\sVec_\RR)$ and $\Mext(\sVec_\HH)$, and show that this does indeed yield some new information regarding these Witt groups.
Our main results are the following theorems:

\begin{reptheorem}{thm:Mext(sVec_R) is Klein-four}
    The group of $\mathcal Mext(\sVec_{\RR})$ of minimal modular extensions of $\sVec_{\RR}$ is isomorphic to the Klein-four group $(\ZZ/2\ZZ)^2$.
    The four classes of extensions are represented by the categories:
    \[
        \TC_{\RR},\hspace{3mm}\TC_{\HH}^b,\hspace{3mm}\TF_{\RR},\hspace{2mm}\text{and}\hspace{2mm}\TF_{\HH}\,,
    \]
    with $\TC_\RR$ playing the role of the identity.
\end{reptheorem}

Here, $\TC$ stands for the toric code, and $\TF$ stands for the 3-fermion theory (see Section \ref{sec:Classification of real ndBFCS of dim=4} for more discussion and notation).
The additional decorations indicate the various real forms that these categories admit.

\begin{reptheorem}{thm:Mext(sVec_H) is Klein-four}
    The group of $\mathcal Mext(\sVec_{\HH})$ of minimal modular extensions of $\sVec_{\HH}$ is isomorphic to the Klein-four group $(\ZZ/2\ZZ)^2$.
    The four classes of extensions are represented by the categories:
    \[
        \TC_{\HH}^f,\hspace{3mm}\TC_{\Tri},\hspace{3mm}\TF_\HH,\hspace{2mm}\text{and}\hspace{2mm}\TF_{\Tri}\,,
    \]
    with $\TC_{\HH}^f$ playing the role of the identity.
\end{reptheorem}

In order to prove these theorems, we fully classify all nondegenerately braided fusion categories over $\RR$ of $\FPdim=4$.
\begin{reptheorem}{thm:List of all ndBFCs over R of dim=4}
    Every nondegenerately braided fusion category $\C$ over $\RR$ of \linebreak$\FPdim(\C)=4$ is a real form of $\TC$, $\DS$ (double semion), or $\TF$.
    The possible real forms are $\TC_\RR$, $\TC_\HH^f$, $\TC_\HH^b$, $\TC_{\Tri}$, $\DS_\RR$, $\TF_\RR$, $\TF_\HH$, and $\TF_{\Tri}$ (See Section \ref{sec:Classification of real ndBFCS of dim=4} for notation).
\end{reptheorem}

Inspired by Lan, Kong, and Wen's surjectivity result \cite[Proposition 5.15]{MR3613518}, we replace $\CC$ with $\RR$ and investigate the corresponding maps $U_\RR$ for $\E=\sVec_\RR$ and $U_\HH$ for $\E=\sVec_\HH$.
In order to understand the kernel of these maps, we fully classify all fusion categories over $\RR$ whose centers have a real unit and Frobenius-Perron dimension 4.

\begin{reptheorem}{thm:All the centers of dim=4}
    Suppose $\C$ is a fusion category over $\RR$, with $\Omega\Z\C:=\End(\1_{\Z(\C)})\cong\RR$ and $\FPdim(\Z(\C))=4$.
    Then the pair $(C,\Z(\C))$ corresponds to exactly one of the rows in the table below.
    \begin{equation}\label{eqn:Table of centers (intro)}
        \begin{array}{c|c|c}
           \text{label} & \C & \Z(\C) \\\hline\\[-8pt]
          (I) & \Vec_\RR(\ZZ/2\ZZ) & \TC_\RR\\[2pt]
          (II) & \C_\HH(1,\chi,-1/2) & \TC_\HH^f\\[2pt]
          (III) & \Vec_\RR^\omega(\ZZ/2\ZZ) & \DS_\RR\\[2pt]
          (IV) & \C_\HH(1,\chi,+1/2) & \DS_\RR\\[2pt]
          (V) & \Vec_\CC\big((\ZZ/2\ZZ)^2_{\text{Gal}}\big) & \TC_\RR\\[2pt]
          (VI) & \Vec_\CC^\alpha\big((\ZZ/2\ZZ)^2_{\text{Gal}}\big) & \TC_\HH^f\\[2pt]
          (VII) & \Vec_\CC\big((\ZZ/4\ZZ)_{\text{Gal}}\big) & \TC_\HH^f\\[2pt]
          (VIII) & \C_{\overline{\CC}}(\ZZ/2\ZZ,\chi) & \TC_{\Tri}\\[2pt]
          (IX) & \Vec_\CC^\omega\big((\ZZ/2\ZZ)^2_{\text{Gal}}\big) & \DS_\RR\\[2pt]
          (X) & \Vec_\CC^{\alpha\omega}\big((\ZZ/2\ZZ)^2_{\text{Gal}}\big) & \DS_\RR\\
        \end{array}
    \end{equation}
\end{reptheorem}

In particular, this shows that $\TC_{\Tri}$ is a Drinfeld center, thus correcting an error in \cite[Appendix A]{MR4934638}.

\begin{repcorollary}{cor:Bosonic TC_H is the true class}
    The category $\TC_\HH^b$\footnote{not $\TC_{\Tri}$ as previously claimed in \cite{MR4934638}} is not a Drinfeld center, and thus represents the unique nontrivial class in
    \[\ker\big(\Witt(\Vec_\RR)\to\Witt(\Vec_\CC)\big)\cong H^4(\RR;\mathbb G_m)\cong\ZZ/2\ZZ\,.\]
\end{repcorollary}

This category $\TC_\HH^b$ is the real form corresponding to a well-known anomalous time-reversal symmetric field theory (see \cite{PhysRevB.87.235122,Kapustin:2014tfa,MR4072236,wen2026categoricaltimereversalsymmetries}).
Corollary \ref{cor:Bosonic TC_H is the true class} indicates that this anomalous theory has further significance in terms of Galois cohomology.

Comparing Table (\ref{eqn:Table of centers (intro)}) with Theorem \ref{thm:Mext(sVec_R) is Klein-four}, we see that only the unit of $\Mext(\sVec_\RR)$ is Witt trivial.

\begin{repcorollary}{cor:No kernel for Mext(sVec_R)}
    The map  $U_{\RR}:\C\mapsto[\C]$ from $\Mext(\sVec_\RR)$ to $\Witt(\Vec_\RR)$ yields an identification
    \[\Mext(\sVec_\RR)\cong\ker\big(\Witt(\Vec_\RR)\to\Witt(\sVec_\RR)\big)\,.\]
\end{repcorollary}

A similar comparison of Table (\ref{eqn:Table of centers (intro)}) with Theorem \ref{thm:Mext(sVec_H) is Klein-four} shows that some of these MNEs are Witt trivial over $\RR$.

\begin{repcorollary}{cor:Some kernel for Mext(sVec_H)}
    The map $U_\HH:\C\mapsto[\C]$ from $\Mext(\sVec_\HH)$ to $\Witt(\Vec_\RR)$ has kernel $\ZZ/2\ZZ$; generated by $\TC_{\Tri}$.
\end{repcorollary}

The presence of $\TC_{\Tri}=\Z(\C_{\overline{\CC}}(\ZZ/2\ZZ,\chi))$ in this kernel can be seen as a consequence of the following fact.

\begin{repproposition}{prop:Nontrivial auto-Witt equivalence for sVec_H}
    The fusion category $\C_{\overline{\CC}}(\ZZ/2\ZZ,\chi)$ is a non-Morita-trivial $\sVec_\HH$-Witt equivalence from $\sVec_\HH$ to itself.
\end{repproposition}

In the final section of the paper, we describe a higher categorical conjecture that organizes all of these ideas.
One potential consequence of this conjecture would be that Corollary \ref{cor:Some kernel for Mext(sVec_H)} and Proposition \ref{prop:Nontrivial auto-Witt equivalence for sVec_H} should imply one another.


\section{Acknowledgements}
This paper answers a question posed to me by Theo Johnson-Freyd at the 2026 Paris workshop for the Simons Collaboration on Global Categorical Symmetries.
I would like to thank Theo for posing the question, and to thank the Simons Foundation for the opportunity to attend the workshop.
Special thanks are due to Thibault Décoppet, who clarified my thinking on $\Mext$ groups.
Some of the arguments used in this paper were taken directly from conversations with Thibault, and I am grateful for the opportunity to present them here.
Finally, thanks are due to Julia Plavnik for suggesting the investigation of Tambara-Yamagami categories, which turned out to be an essential ingredient of many proofs in this paper.

The tool \href{Connected Papers}{https://www.connectedpapers.com} was used to find references to minimal nondegenerate extensions in the literature.
Generative AI text was not used in the preparation of this manuscript, nor was it used during the research phase of the project.

Research for this paper was funded, in parts, by the Simons Collaboration on Global Categorical Symmetries (award number 888988), the Engineering and Physical Sciences Research Council Open Fellowship
“Complex Quantum Topology” (grant number EP/Y008812/1), and the European Research Council grant on Non-compact Chern-Simons Theory, Positive Representations, and Cluster Varieties(grant agreement ID: 948885, \href{https://doi.org/10.3030/948885}{https://doi.org/10.3030/948885}).


\section{Some facts about fusion categories over \texorpdfstring{$\RR$}{R}.}

\subsection{Basics}

We assume that the reader is familiar with basic concepts from the theory of braided tensor categories.
Familiarity with Chapters 3 and 8 of \cite{MR3242743} will be helpful.

Our main objects of study are $\sVec_\RR$ and $\sVec_\HH$, which are fusion categories over $\RR$.
Fusion categories over non-algebraically closed fields were investigated in \cite{MR4806973}, and we will adopt those conventions.
Section 3 of \cite{MR5003359} offers an introduction to fusion categories over $\RR$ specifically.
For the more physically-minded reader, the recent paper \cite{wen2026categoricaltimereversalsymmetries} discusses many real fusion categories, as well as their relation with time-reversal symmetry.

Schur's lemma implies that the endomorphism algebra of any simple object is a division algebra.
Over $\RR$, the finite dimensional division algebras are $\RR$, $\CC$, and $\HH$.
When a simple object has this as its division algebra, we will say that the simple object is real, complex, or quaternionic, respectively.
In general, a simple object $X$ with $\End(X)$ equal to the base field are called split, or split-simple.

All monoidal categories $\C$ have a preferred object, the monoidal unit $\1$.
This object can be viewed as a pointing, that is to say, a functor from the trivial category to $\C$.
In analogy with algebraic topology, we will write $\Omega\C$ for the endomorphisms of the pointing.

\begin{definition}
    The endomorphism algebra $\End(\1)$ of the unit object $\1$ in a monoidal category $\C$ will be denoted $\Omega\C$.
\end{definition}

The endomorphisms of the unit are always commutative, by the Eckmann-Hilton argument, so when $\C$ is fusion over $\mathbb R$, $\Omega\C$ is either $\RR$ or $\CC$.

For any object $X$, there are left and right algebra embeddings of $\Omega\C$ into $\End(X)$.
When $\Omega\C=\RR$, this is just the normal scalar product.
When $\Omega\C\cong\CC$, since we are in an $\RR$-linear setting, it is possible for the left and right embeddings to differ by complex conjugation.
When this happens, we say that the object $X$ is Galois nontrivial (see \cite[Section 3.2]{MR5003359} for a more detailed account).

The Deligne tensor product $\boxtimes=\boxtimes_{\RR}$ of simple objects over $\RR$ is controlled by the Artin-Wedderburn decomposition of tensor products of algebras over $\RR$, shown below.
\[
    \begin{array}{c|c|c|c|}
        \otimes_\RR & \RR & \CC & \HH \\\hline
        \RR & \RR & \CC & \HH \\\hline
        \CC & \CC & \CC\oplus\CC & M_2(\CC) \\\hline
        \HH & \HH & M_2(\CC) & M_4(\RR)\\\hline
    \end{array}
\]
The above table shows that the Deligne product of two complex simples must decompose into two nonisomorphic simples, and the Deligne product of two quaternionic simples must decompose into four copies of a unique simple object.

\begin{example}
    The category $\sVec_\HH$ has two simple objects, $\1$ and $Y$.  The unit $\1$ is real, and the other simple $Y$ is quaternionic.
    In the Deligne product $\sVec_\HH\boxtimes\Vec_\CC$, the object $Y\boxtimes\CC$ is not simple, but decomposes into two copies of a unique simple object.
    Of course, this product category is just the complexification of $\sVec_\HH$, which is $\sVec_\CC$, and the simple object that appears inside of $Y\boxtimes\CC$ is just the odd line.
\end{example}

\begin{definition}
    A simple object $Y$ in a fusion category $\C$ over $\KK$ is said to be quasi-invertible if $Y\otimes Y^*\cong n\cdot\1$, for some positive integer $n$.
\end{definition}

\begin{example}
    The object $Y\in \sVec_{\HH}$ is self-dual, and satisfies $Y\otimes Y\cong4\cdot\1$, so it is quasi-invertible.
\end{example}

This notion of quasi-invertible object and the upcoming lemma seem to be a minor observation, but they end up being very convenient for analyzing fusion rings.
At the moment, we are unaware of this appearing in the literature, though we expect that it was known to experts.

\begin{lemma}\label{lem:Quasi-invertible rules}
    For any quasi-invertible $Y\in\C$, the algebra $\End(Y)$ is central simple over $\Omega\C$, and for any other object $X\in\C$, 
    \[\End(X\otimes Y)\cong\End(X)\otimes_{\Omega\C}\End(Y)\,.\]
\end{lemma}

\begin{proof}
    For any object $X\in\C$, 
    \[\End(X\otimes Y)\cong\Hom(X,X\otimes Y\otimes Y^*)\cong\Hom(X,n\cdot X)\cong n\cdot\End(X)\,,\]
    as vector spaces over $\Omega\C$.
    In particular, setting $X=\1$ shows that $\End(Y)$ has dimension $n$ over $\Omega\C$.
    For any pair of objects $A,B\in\C$, there is an inclusion of algebras
    \[\End(A)\otimes_{\Omega\C}\End(B)\hookrightarrow\End(A\otimes B)\,.\]
    For $(A,B)=(Y,Y^*)$, we can use the algebra inclusion to find that
    \begin{gather*}
        \End(Y)\otimes_{\Omega\C}\End(Y)^{op}\cong\End(Y)\otimes_{\Omega\C}\End(Y^*)\\
        \hookrightarrow\End(Y\otimes Y^*)\cong\End(n\cdot\1)\cong M_n(\Omega\C)\,.
    \end{gather*}
    Comparing dimensions, we find that this must be an isomorphism of algebras.
    If $Z(\End(Y))$ were strictly larger than $\Omega\C$, then $\End(Y)\otimes_{\Omega\C}\End(Y)^{op}$ would decompose into multiple matrix blocks, from whence it follows that $\End(Y)$ must be central simple over $\Omega\C$.

    For $(A,B)=(X,Y)$, the inclusion
    \begin{gather*}
        \End(X)\otimes_{\Omega\C}\End(Y)\hookrightarrow\End(X\otimes Y)
    \end{gather*}
    must again be an isomorphism, for dimension reasons.
\end{proof}

\begin{corollary}\label{cor:Complex simples are big}
    If $\Omega\C=\RR$ and $X$ is a complex simple object, then $\FPdim(X)>\sqrt2$.
\end{corollary}

\begin{proof}
    Rigidity forces $X\otimes X^*\cong 2\cdot\1\oplus V$, for some object $V\in \C$.
    By Lemma \ref{lem:Quasi-invertible rules}, $V$ must be nonzero, and so $\FPdim(X)^2=2+\FPdim(V)>2$.
\end{proof}

The computation of Frobenius-Perron dimensions of objects remains the same as in the algebraically closed case (see \cite[Chapter 3]{MR3242743}).
However, the presence of non-split simple objects causes the Frobenius-Perron dimension \emph{of categories} to have a different formula.

\begin{definition}[{\cite[Definition 3.21]{MR4806973}}]
    For $\C$ a fusion category over a field $\KK$, with $\End(\1)=\Omega\C$, the Frobenius-Perron dimension of $\C$ is the quantity
    \begin{equation}\label{eqn:FPdim formula}
        \FPdim(\C)=\sum_{\substack{X\in\C,\\ X \text{ simple}}}\frac{\FPdim(X)^2}{\dim_{\Omega\C}\big(\End(X)\big)}\;.
    \end{equation}
\end{definition}

Similar denominators appear in Kevin Walker's universal state sum TQFT formulas \cite{walker2021universalstatesum}.
The universality of this pattern of normalization indicates that the lack of denominators in the $\FPdim$ formula over the complex numbers is simply a consequence of $\mathbb C$ being algebraically closed.

\begin{example}
    When $H$ is a semisimple Hopf algebra over a field $\KK$,
    \[\FPdim(\Rep_{\mathbb K}(H))=\dim_{\mathbb K}(H)\,.\]
\end{example}

\begin{example}
    The object $Y\in\sVec_\HH$ satisfies $Y\otimes Y\cong 4\cdot\1$, and thus has $\FPdim(Y)=2$.
    From this it follows that
    \[\FPdim(\sVec_\HH)=\frac{\FPdim(\1)^2}{\dim_{\Omega\C}(\Omega\C)}+\frac{\FPdim(Y)^2}{\dim_{\Omega\C}(\End(Y))}\;=\;2\,.\]
\end{example}

Note that the second summand in the above example was still an integer, despite the denominator.
This is actually a more general phenomenon that we will take advantage of when classifying centers in Section \ref{sec:Classification of centers}.

\begin{proposition}[{cf. \cite[Theorem 3.38]{MR4806973}}]\label{prop:Each simple contributes at least 1 to FPdim}
    Each summand
    \[\frac{\FPdim(X)^2}{\dim_{\Omega\C}(\End(X))}\]
    in Equation \ref{eqn:FPdim formula} is the Frobenius-Perron dimension of the internal object
    \[X^*\otimes_{\End(X)}X\,,\]
    and is therefore an algebraic integer that that is $\geq1$.
\end{proposition}

\begin{example}\label{eg:Complex simples contribute more}
    When $\Omega\C=\RR$ and $X$ is a complex simple, Corollary \ref{cor:Complex simples are big} implies that
    \[\frac{\FPdim(X)^2}{\dim_{\Omega\C}(\End(X))}>\frac{2}{2}\;=\;1\,\]
    and so $X$ must contribute strictly more than 1 to $\FPdim(\C)$.
\end{example}

\subsection{Galois descent}

Here we give the basics of categorical descent as outlined by Etingof and Gelaki in \cite{MR2946231}.
This theory involves maps between higher groupoids attached to fusion categories such as $\Pic(\C)$ and $\Aut_{br}(\C)$, as well as their classifying spaces.
A good reference for such constructions is \cite{MR4281262}.

Given a real fusion category $\C$, we can extend scalars to $\CC$ by taking the Deligne product $\overline\C:=\Vec_\CC\boxtimes_\RR\C$.
When $\Omega\C=\RR$, $\overline\C$ will be fusion over $\CC$.
If $\Omega\C=\CC$, then $\overline\C$ will be multifusion.
If $\C$ had a braiding, then this will extend to a braiding on $\overline\C$.
Additionally, extension of scalars commutes with the operation of taking centers:
\[\Z(\overline\C)\simeq\overline{\Z(\C)}\,.\]

Many questions about $\C$ can be answered by checking properties of $\overline\C$, and vice versa.
For example, in the case where $\Omega\C=\CC$, $\overline\C$ is indecomposable if and only if $\C$ has Galois nontrivial objects, and moreover this is equivalent to $\Omega\Z\C=\RR$.

If $\D$ is fusion over $\CC$, then a fusion category $\C$ over $\RR$ is said to be a real form of $\D$ if $\overline\C\simeq\D$.
The data of such an equivalence is called a framing.
Two real forms are equivalent if they are equivalent as real categories.
An equivalence that also respects the framings is called a framed equivalence.
The braided versions follow the same terminology.

The categorical descent theory of \cite{MR2946231} tells us that that framed real forms of $\D$ are in bijection with homotopy classes of continuous maps
    \[B\ZZ/2\ZZ\to B\Aut_{\RR,\otimes}(\D)\footnote{Etingof and Gelaki prefer to think of `Galois-twisted' autoequivalences (cf. \cite[Section 3.2]{MR2946231}), but this is only a matter of perspective, and not of mathematical distinction.}\,,\]
with the property that the generator of $\ZZ/2\ZZ$ acts by some monoidal functor $(T,J)$ that performs complex conjugation on $\Omega\D$, and where $\Aut_{\RR}(\D)$ is the categorical group of $\RR$-linear autoequivalences of $\D$.
Braided real forms follow the same pattern, with braided autoequivalences $\Aut_{\RR,br}(\D)$ replacing the target.

\begin{example}\label{eg:Semion has no braided real forms}
    The semion category $\Sem$ is the braided category $\Vec_{\CC}^{(\omega,\beta)}(\ZZ/2\ZZ)=\langle\1, \sigma\rangle_{\oplus}$, with associator and braiding determined by the formulas 
    \begin{gather*}
        \Big(\omega_{\sigma,\sigma,\sigma}:(\sigma\otimes\sigma)\otimes\sigma\to\sigma(\sigma\otimes\sigma)\Big)=-\id_\sigma\,,\text{ and}\\
        \Big(\beta_{\sigma,\sigma}:\sigma\otimes\sigma\to\sigma\otimes\sigma\Big)=i\cdot\id_\1.
    \end{gather*}
    Any monoidal autoequivalence $T$ of $\Sem$ must fix all objects, and the monoidal structure $J$ can be assumed to be trivial, because $H^2(\ZZ/2\ZZ;\CC^\times)=0$.
    The presence of $i$ in the formula for the braiding implies that any autoequivalence that performs complex conjugation would not be braided, so $\Sem$ has no braided real forms.
\end{example}

The process of finding a continuous map that determines a real form can be carried out in steps.
First we choose a functor $(T,J)$, that acts by complex conjugation and where $(T,J)^2$ is isomorphic to the identity.
Next we choose such a natural isomorphism $\mu:(T,J)^2\to\id_\D$.
This determines an obstruction cocycle in $Z^3(\ZZ/2\ZZ;\Aut_{\otimes}(\id_\D))$, where the coefficient module $\Aut_{\otimes}(\id_\D)$ has a twisted action coming from conjugation by $T$.
The homology class of this cocycle is denoted $O_3(T)$, because it depends on $T$ but not $\mu$.
If the obstruction class is zero, then a real form does exist, and the framed real forms built out of $(T,J)$ are determined by the different choices for $\mu$, and are thus in bijection with $H^2(\ZZ/2\ZZ;\Aut_{\otimes}(\id_\D))$.
Explicitly, the real form can described as the equivariantization $\C^{\ZZ/2\ZZ}$ with respect to the action of $\ZZ/2\ZZ=\mathrm{Gal}(\CC/\RR)$ given by $(T,J,\mu)$ (see \cite[Definition 2.7.2]{MR3242743}).

\begin{remark}
    Typically there is no preferred isomorphism with $H^2$.
    However, if a preferred real form $\C$ exists, then we can identify that with $0\in H^2(\ZZ/2\ZZ;\Aut_{\otimes}(\id_\D))$, and compare all other options for $\mu$ relative to this preferred choice $\mu_\C$.
\end{remark}

\section{Classification of real nondegenerate categories of \texorpdfstring{$\FPdim=$}{dimension} 4}\label{sec:Classification of real ndBFCS of dim=4}

Since $\FPdim(\C)=\FPdim(\overline\C)$, and our MNEs must have $\FPdim=4$, all of our MNEs will complexify to become nondegenerate braided categories of dimension 4.
This means that we will need to have a working knowledge of various complex fusion categories with these properties, as well as notation for their real forms.
The property of having a real form will turn out to be very restrictive, and this will allow us to classify all possibilities.

A fusion category is said to be \emph{pointed} if all of its simple objects are invertible.
Pointed braided fusion categories over $\CC$ are classified by pre-metric groups, that is, finite abelian groups $A$, equipped with a quadratic form $q:A\to\CC^\times$.
A pre-metric group is called a metric group when the quadratic form is nondegenerate.
As mentioned in the introduction, this classification goes back to Eilenberg and Mac Lane \cite{MR65162}, but a convenient reference for this theory can be found in \cite[Appendix A]{MR2609644}.

\begin{theorem}[{cf. \cite[Proposition 2.41]{MR2609644}}]\label{thm:Equivalence of pre-metric groups and PBFCs}
    There is an equivalence of categories between pre-metric groups, and the 1-categorical truncation of the 2-category of pointed braided fusion categories, braided functors, and monoidal natural transformations.
\end{theorem}

Concretely, this means that pointed braided fusion categories are essentially just pre-metric groups, and all braided functors are determined up to isomorphism by their underlying map of of pre-metric groups.
We will denote the pointed braided fusion category associated to $(A,q)$, by the notation $\C(A,q)$.
The braiding on $\C(A,q)$ is nondegenerate if and only if the quadratic form $q$ is nondegenerate.

\begin{definition}\label{def:Gauss sum and central charge}
    The Gauss sum $\tau(A,q)$ of a pre-metric group $(A,q)$ is the complex scalar
    \[\tau(A,q)=\sum_{a\in A}q(a)\;.\]
    If $\tau(A,q)\neq0$, the central charge is the normalized Gauss sum
    \[c(A,q):=\frac{\tau(A,q)}{|\tau(A,q)|}\;.\]
    If $\C=\C(A,q)$, then we will use the shorthand $\tau(\C):=\tau(A,q)$ and $c(\C):=c(A,q)$.
\end{definition}
The Gauss sum and central charge are multiplicative with respect to $\boxtimes$, and $\tau(\C)=0$ if and only if $\C$ is degenerate.

The following three examples will play an important role in our story.

\begin{example}\label{eg:Introducing TC}
    The toric code is the braided category $\TC=\Z(\Vec_{\CC}(\ZZ/2\ZZ))$.
    This category is pointed braided, of the form $\C((\ZZ/2\ZZ)^2,q)$, where $q(a)=-1$ for exactly one $a\in(\ZZ/2\ZZ)^2$, and $q(x)=1$ for the other three elements.
    The Gauss sum is $\tau(\TC)=2$.
\end{example}

\begin{example}\label{eg:Introducing DS}
    The category $\Z(\Vec_\CC^\omega(\ZZ/2\ZZ))$ is equivalent to $\Sem\boxtimes\Sem^{rev}$ (see Example \ref{eg:Semion has no braided real forms}), and is thus referred to as the double-semion category.
    We will denote this category by $\DS$.
    This category is pointed, of the form $\C((\ZZ/2\ZZ)^2,q)$, where $q(\sigma)=i$, $q(\overline\sigma)=-i$, and $q(\sigma\overline\sigma)=1$.
    The Gauss sum is $\tau(\DS)=2$.
\end{example}

\begin{example}\label{eg:Introducing TF}
    The three fermion category $\TF$ is $\C((\ZZ/2\ZZ)^2,q)$, where $q(a)=-1$ for all three nontrivial elements of $(\ZZ/2\ZZ)^2$.
    The Gauss sum is $\tau(\TF)=-2$.
\end{example}

Before continuing on, let's back up a bit and clarify some terminology from physics.

\begin{definition}
    For a simple object $X$, the behavior of its self-braiding $\beta_{X,X}:X\otimes X\to X\otimes X$ is called its exchange statistics.
    The self braiding can be complicated in general, but when it is a scalar multiple of the identity, there are special names.
    When $\beta_{X,X}=\id_{X\otimes X}$, we say that $X$ has bosonic exchange statistics, or that $X$ is a boson.
    When $\beta_{X,X}=-\id_{X\otimes X}$, we say that $X$ has fermionic exchange statistics, or that $X$ is a fermion.
    The name semion is meant to evoke the idea of having exchange statistics that are `halfway in between' $1$ and $-1$.
    More generally, if $\beta_{X,X}=\lambda\cdot\id_{X\otimes X}$, where $\lambda$ is some other complex phase, $X$ is called an anyon.
    
    For our purposes, real or quaternionic simple objects $X$ will be called fermions or bosons based on the exchange statistics of the unique simple object inside of $X\boxtimes\CC$.
\end{definition}

\begin{example}\leavevmode
    \begin{itemize}
        \item The toric code $\TC$ has three bosons and one fermion.
        \item The double semion category $\DS$ has two bosons and two semions.
        \item The three fermion category $\TF$ has (unsurprisingly) three fermions and one boson.
        \item The object $Y$ in $\sVec_\HH$ is a quaternionic fermion.
    \end{itemize}
\end{example}

When complex scalars are conjugated, the category $\C(A,q)$ becomes $\C(A,\overline q)$.
If $\C(A,q)\simeq\overline\C$ for some braided real fusion category $\C$, then $\C(A,q)$ admits a braided equivalence to its complex conjugate.
By Theorem \ref{thm:Equivalence of pre-metric groups and PBFCs}, such an equivalence must come from an isomorphism $(A,q)\to(A,\overline q)$ of pre-metric groups.
Since $\tau(A,\overline q)=\overline{\tau(A,q)}$, it follows that whenever $\C(A,q)$ admits a real form, $\tau(A,q)\in\RR$.

Fusion categories over $\CC$ of $\FPdim=4$ that are not pointed are called \emph{Ising} categories.
Up to monoidal equivalence, there are only two such categories, and they have $\ZZ/2\ZZ$ Tambara-Yamagami fusion rules.
All braidings on Ising categories are nondegenerate, and none of them admit real forms.
Thus, if a nondegenerately braided category of $\FPdim=4$ admits a real form, it must be pointed.

\begin{proposition}\label{prop:The only possible complexifications}
    Suppose $\C$ is a nondegenerately braided fusion category over $\RR$, with $\FPdim(\C)=4$.
    The complexification $\overline\C$ of $\C$ is pointed, and $\tau(\overline\C)$ is real.
    If $\tau(\overline\C)>0$, then $\overline\C$ is a Drinfeld center, either $\TC$ or $\DS$.
    If $\tau(\overline\C)<0$, then $\overline\C$ is $\TF$.
\end{proposition}

\begin{proof}
    The first two statements are just a rephrasing of the previous discussion, so it will suffice to prove the other two claims.
    
    If $\tau(\overline\C)>0$, then by \cite[Proposition A.7]{MR2609644}, $\overline\C$ has a Lagrangian subgroup, which implies that $\overline\C$ is a Drinfeld center.
    Any complex category $\D$ for which $\Z(\D)\simeq\overline\C$ must have $\FPdim(\D)=2$, which forces $\D$ to be either $\Vec_\CC(\ZZ/2\ZZ)$ or $\Vec_\CC^\omega(\ZZ/2\ZZ)$.
    The first corresponds to $\TC$, and the second to $\DS$.

    Suppose that $\tau(\overline\C)<0$.
    If $A=\ZZ/4\ZZ=\langle x\rangle$, then $q(x)^2=b(x,x)$ must be a 4th root of unity, and so $q(x^k)=\xi^{k^2}$ for some 8th root of unity $\xi$.
    The Gauss sum is then $\tau(\overline\C)=1+\xi+\xi^4+\xi$, but this can never be negative.
    If $A=(\ZZ/2\ZZ)^2=\langle x,y\rangle$, then $q(x)^2=b(x,x)=\pm1$, and the same for $y$.
    Choosing $b(x,y)=\pm1$, $q(x)=i^j$ and $q(y)=i^k$ gives $q(x+y)=b(x,y)q(x)q(y)=\pm i^{j+k}$.
    The Gauss sum is then $\tau(\overline\C)=1+i^j+i^k\pm i^{j+k}$.
    The only way for this to be negative is if $j=k=2$ and $b(x,y)=-1$.
    This corresponds to $\TF$, and the proof is complete.
\end{proof}

Since there are so few of these categories, let us proceed to classify all of them.
We will treat $\TC$ and $\TF$ simultaneously, as the proofs are very similar.

\begin{proposition}\label{prop:Classification of real forms of TC and 3F}
    The category $\TC$ admits four real forms, which we will denote by $\TC_\RR$, $\TC_\HH^f$, $\TC_\HH^b$, and $\TC_{\Tri}$.
    The category $\TF$ admits three real forms, which we will denote by $\TF_\RR$, $\TF_\HH$, $\TF_{\Tri}$.
    
\end{proposition}

Here, the subscript $\RR$ means that all simple objects are real, and the subscript $\HH$ indicates that there are two quaternionic objects.
The superscript $f$ indicates that \emph{one of} the quaternionic objects is a fermion, while the superscript $b$ indicates that \emph{both} quaternionic objects are bosons.
The subscript $\Tri$ indicates that there is exactly one real, one quaternionic, and one complex simple object, so will call these categories triumvirates.

\begin{proof}
    Both $\TC$ and $\TF$ admit a skeletal description (gauge) where all associator coefficients are trivial, and where all braiding coefficients are $\pm1$.
    This implies that the tensorator $J$ of the functor $T$ is classified by an element of (a torsor of) $H^2_{br}((\ZZ/2\ZZ)^2;\CC^\times)=0$ (here the coefficients are untwisted).
    It follows that $J$ is unique up to monoidal equivalence, when it exists\footnote{The reader should think of this observation as an invocation of Theorem \ref{thm:Equivalence of pre-metric groups and PBFCs}.
    This cohomological fact is what implies that braided $\CC$-linear equivalences of metric group categories are determined up to isomorphism by the underlying group automorphism.
    In this particular context, complex conjugation doesn't affect the result, because all coefficients are real.}.

    One possible choice of $T$ is just the complex conjugation functor, which admits a trivial tensorator $J$.
    Possible real forms resulting from this choice are the braided category $\TC_\RR:=\mathcal Z(\Vec_{\RR}(\ZZ/2\ZZ))$, or the purely real version $3\mathrm{F}_{\RR}$ of $\TF$.
    This shows that the $O_3$ obstruction for this $T$ is trivial for both $\TC$ and $\TF$.
    The different choices of framed real forms that use this choice of $(T,J)$ form a torsor over
    \[H^2(\ZZ/2\ZZ;\Aut_{\otimes}(\id_{\D}))\cong H^2(\ZZ/2\ZZ;(\ZZ/2\ZZ)^2)\cong(\ZZ/2\ZZ)^2\,.\]
    These classes determine which two simples in the real form are quaternionic, if any.
    
    Since $\TC$ has a unique fermionic simple, exactly two of the cohomology classes described above correspond to real forms where the fermion becomes quaternionic, and hence these two forms contain $\sVec_{\HH}$.
    The autoequivalence that swaps the two bosons permutes these two cohomology classes, and so the resulting categories are braided equivalent to one another, and we will denote them by $\TC_\HH^f$.
    The trivial cohomology class corresponds to the form $\TC_\RR$, while the remaining class corresponds to the form where the fermion is real.
    This last class will be denoted $\TC_\HH^b$, because both quaternionic simples are bosonic.

    In the $\TF$ case, all three fermions are in the same orbit under braided autoequivalences, and so all of the nontrivial cohomology classes produce equivalent braided categories, and we will denote the resulting real form by $\TF_\HH$.
    The trivial cohomology class causes no Brauer-twistings, and so all four simple objects are real, and we will call this class $\TF_\RR$.

    The other option for $T$ is a functor that swaps two simple objects. 
    Note that $\TF$ has other braided autoequivalences, but swapping is the only $T$ that will have order two, up to conjugacy and isomorphism.
    For $\TF$, it doesn't matter which two fermions are swapped, but in the case of $\TC$, a braided functor must fix the unique fermion, so $T$ must swap the two bosons with each other.

    It was already shown in \cite[Appendix A]{MR4934638} that $\TC$ has a real form corresponding to the swap functor, and this form contains a copy of $\sVec_{\HH}$.
    This category has exactly one simple of each type, so we will call it a \emph{triumvirate} and denote it by $\TC_{\Tri}$.
    For both $\TC$ and $\TF$, the swapping action makes the $\ZZ/2\ZZ$ module $\Aut_{\otimes}(\id_\D)\cong(\ZZ/2\ZZ)^2$ acyclic, so in particular, the real form corresponding to the swap functor is unique, assuming it exists.

    We have found all the real forms of $\TC$, so we will focus on $\TF$ for the rest of the proof.
    Let the fermions be $a$, $b$, and $ab$.
    The functor $T$ will swap $a$ and $b$.
    We can choose a gauge where the braiding is given by
    \[\beta_{a^ib^j,a^kb^\ell}=(-1)^{ik+i\ell+j\ell}\cdot\id_{a^{i+k}b^{j+\ell}}\,.\]
    For the tensorator $J$, we can use the formula
    \[J_{a^ib^j,a^kb^\ell}=(-1)^{jk}\cdot\id_{a^{j+\ell}b^{i+k}}\,.\]
    Recall that if this choice works, then it is unique up to coboundaries.
    This is a valid tensorator, since it is a 2-cocycle and the associators are trivial.
    This choice is braided, as can be seen by the commutativity of the diagram below:
    \[\begin{tikzcd}[ampersand replacement=\&]
    	{T(a^ib^j)\otimes T(a^kb^\ell)} \&\& {T(a^ib^j\otimes a^kb^\ell)} \\
    	{T(a^kb^\ell)\otimes T(a^ib^j)} \&\& {T(a^kb^\ell\otimes a^ib^j)}
    	\arrow["{(-1)^{jk}}", from=1-1, to=1-3]
    	\arrow["{(-1)^{j\ell+jk+ik}}"', from=1-1, to=2-1]
    	\arrow["{(-1)^{ik+i\ell+j\ell}}", from=1-3, to=2-3]
    	\arrow["{(-1)^{\ell i}}", from=2-1, to=2-3]
    \end{tikzcd}\;\;.\]

    The next piece of information we need is a monoidal natural isomorphism $\mu:T^2\to\id$.
    The underlying functor of $T^2$ is the identity, and the tensorator is
    \[J^2_{a^ib^j,a^kb^\ell}=(-1)^{i\ell+jk}\cdot\id_{a^{i+k}b^{j+\ell}}\,.\]
    The map
    \[\mu_{a^ib^j}=(-1)^{ij}\cdot\id_{a^ib^j}\]
    provides the desired natural isomorphism, because $\delta\mu=J^2$.
    This $\mu$ provides two isomorphisms from $T^3$ to $T$: $\mu\circ\id_T$ and $\id_T\circ\mu$.
    The obstruction $O_3$ is determined by the ratio of the corresponding scalars.
    It follows that
    \[O_3(1,1,1)=\frac{(-1)^{ji}}{(-1)^{ij}}=1\,,\]
    so the obstruction vanishes, and a real form exists.
    The swapping action creates a complex object $(a\oplus b)$, and the fact that $\mu(ab)=-1\cdot\id_{ab}$ means that $ab$ becomes quaternionic in the real form.
    With one real, one complex, and one quaternionic object, this is another triumvirate category, and we denote it by $\TF_{\Tri}$.

    In summary, there were two options for $T$: all objects are fixed, or two objects are swapped.
    There was never any choice for $J$.
    For $\TC$, when $T$ fixes objects, there were 4 options for $\mu$ that resulted in $\TC_\RR$, $\TC_\HH^f$($\times2$), and $\TC_\HH^b$, respectively.
    For $\TF$, when $T$ fixes objects, there were 4 options for $\mu$ that resulted in $\TF_\RR$, and $\TF_\HH^f$($\times3$), respectively.
    For both $\TC$ and $\TF$, when $T$ swaps two objects, the unique choice for $\mu$ produces $\TC_{\Tri}$ and $\TF_{\Tri}$, respectively.
\end{proof}

\begin{remark}
    In the above proof, the scalar $\mu(ab)=-1\cdot\id_{ab}$ can be seen as a generalization of the of the Frobenius-Schur indicator.
    This anomaly is precisely what causes $ab$ to become quaternionic in the real form.
    Any attempt to endow $ab$ with an equivariant structure $T(ab)\to ab$ will lead to the equation $|\lambda|^2=\lambda T(\lambda)=\mu(ab)=-1$.
    To get around this, take the direct sum $ab\oplus ab$, and define the equivariant structure to be any $2\times2$ matrix $M$ such that $M\circ T(M)=-I$.
    This works, and the resulting equivariant morphisms come from the copy of $\HH$ inside of $M_2(\RR)$.
\end{remark}

Because the double semion category has no fermionic objects, its real form will not appear in the $\Mext$ groups we plan to investigate.
However, it will appear in Section \ref{sec:Classification of centers}, so it is worth computing.

\begin{proposition}\label{prop:DS has a unique real form}
    The double semion category $\DS=\Z(\Vec_\CC^\omega(\ZZ/2\ZZ))$ has exactly one equivalence class of real form.
    This category $\DS_\RR$ is $\ZZ/2\ZZ$-graded, with trivial component $\Vec_\RR(\ZZ/2\ZZ)=\langle\1,\sigma\overline\sigma\rangle_{\oplus}$, and nontrivial component $\Vec_\CC=\langle(\sigma\oplus\overline\sigma)\rangle_\oplus$.
    The real object $\sigma\overline\sigma$ is bosonic.
\end{proposition}

\begin{proof}
    As we saw in Example \ref{eg:Semion has no braided real forms}, the category $\Sem$ has no braided real forms, because a conjugating functor replaces the $i$-braiding with a $-i$-braiding.
    In the double semion category, there is an object $\overline\sigma$, coming from $\Sem^{rev}$, that does have such a $-i$-braiding.
    Thus, in order for $T$ to be braided, it must swap $\sigma$ and $\overline\sigma$.
    The monoidal structures $J$ that we can equip $T$ with form a torsor over $H^2((\ZZ/2\ZZ)^2;\CC^\times)\cong\ZZ/2\ZZ$ (untwisted).
    A representative cocycle for this nontrivial cohomology class is $J_{\sigma^i\overline\sigma^j,\sigma^k\overline\sigma^\ell}=(-1)^{i\ell}\cdot\id$.
    When this cocycle is placed on the (non-conjugating) identity functor, it does not satisfy the braiding condition, and it follows that for any braided functor $(T,J)$, the tensorator $J$ is unique up to cohomology\footnote{Unlike the previous cases in Proposition \ref{prop:Classification of real forms of TC and 3F}, the braiding coefficients are no longer real, and so this must be checked directly.}.

    For $T$ that conjugates scalars and swaps $\sigma$ and $\overline\sigma$, the trivial tensorator $J=\id$ is braided, and so this must be the unique choice.
    Since the tensorator is trivial, $T^2=\id$ on the nose, and we can choose $\mu$ to be the identity natural transformation.
    Strictness of composition, together with $\mu=\id$ allows us to easily verify that $O_3=0$, so a real form does exist.
    As in the proof of Proposition \ref{prop:Classification of real forms of TC and 3F}, the swapping action of $T$ makes $\Aut_{\otimes}(\id_\DS)$ acyclic, and so $H^2(\ZZ/2\ZZ;\Aut_{\otimes}(\id_\DS))=0$, and the real form is unique.

    The real form can be constructed as the equivariantization $\DS_\RR=\DS^{\ZZ/2\ZZ}$.
    The objects $\1$ and $\sigma\overline\sigma$ are easily seen to admit equivariant structures coming from identity maps, and the object $\sigma\oplus\overline\sigma$ also admits an equivariant structure via swapping.
    Galois equivariant endomorphisms of $\1$ and $\sigma\overline\sigma$ must be real, and endomorphisms of $(\sigma\oplus\overline\sigma)$ must be of the form $\lambda\oplus\overline\lambda$.
    The braiding is induced from $\DS$, and so $\sigma\overline\sigma$ remains bosonic.
\end{proof}

We conclude with a theorem that summarizes all the categories found in this section.

\begin{theorem}\label{thm:List of all ndBFCs over R of dim=4}
    Every nondegenerately braided fusion category $\C$ over $\RR$ of $\linebreak\FPdim(\C)=4$ is a real form of $\TC$, $\DS$, or $\TF$.
    The possible real forms are $\TC_\RR$, $\TC_\HH^f$, $\TC_\HH^b$, $\TC_{\Tri}$, $\DS_\RR$, $\TF_\RR$, $\TF_\HH$, and $\TF_{\Tri}$.
\end{theorem}


\section{Classification of Mext groups}\label{sec:Classification of Mext groups}

Here we assume knowledge of the theory developed in \cite{MR2609644}, \cite{MR3039775} and \cite{MR3022755}.
Specifically, centers, local modules, and Witt equivalence will all be used.

For the moment, consider a general symmetric fusion category $\E$ over $\RR$.
The tensor product functor $\otimes:\E\boxtimes\E\to\E$ is monoidal, and has a right adjoint $I:\E\to\E\boxtimes\E$.
The algebra $A:=I(\1)$ is sometimes called\footnote{Both $R(\1)\in\Z(\E)$ and $I(\1)\in\E\boxtimes\E$ are canonical algebras associated to $\E$.  They can be interpreted as internal endomorphisms of $\1\in\C$ with respect to the module structures $\Z(\E)\curvearrowright\E$ $\E\boxtimes\E^{op}\curvearrowright\E$, respectively.} the canonical algebra of $\E$.

For any braided category $\C$, an algebra $A$ is called étale if it is commutative, separable, and connected (connected means $\C(\1,A)\cong\Omega\C$).
The category of right modules $\C_A$ for an étale algebra $A$ is a fusion category, because commutativity allows for $\otimes_A$, separability enforces semisimplicity, and connectedness implies that the unit $A$ is a simple module.
All modules that braid trivially with $A$ are said to be \emph{local}, and the category of local modules $\C_A^0$ forms a full subcategory of $\C_A$.
The left action $\C\curvearrowright\C_A$ and the right action $\C_A\curvearrowleft\C_A^0$ are both central, and the images of $\C$ and $\C_A^0$ centralize each other inside of $\Z(\C_A)$.
In this way, we obtain a Witt equivalence $\C\boxtimes(\C_A^0)^{rev}\simeq\Z(\C_A)$.
For more on this, see \cite{MR3022755}.

\begin{definition}
    Suppose $\E$ a symmetric fusion category over $\RR$.
    If $\E\hookrightarrow\C$ and $\E\hookrightarrow\D$ are two minimal nondegenerate extensions, then their product is defined to be
    \[\E\hookrightarrow(\C\boxtimes\D)_{A}^0\;=:\;\C\boxdot\D\;,\]
    where $A=I(\1)\in \E\boxtimes \E$ is the canonical algebra, and the notation $(-)_A^0$ denotes the category of local modules for $A$ (see \cite{MR3613518} for more details).
    The category of ordinary modules $(\C\boxtimes\D)_A$ is equivalent to $\C\boxtimes_{\E}\D$, and contains the local modules as a fusion subcategory.
\end{definition}

For pointed braided fusion categories $\C$ over $\CC$, étale algebras correspond to isotropic subgroups, and local modules are the corresponding isotropic reduction.
It follows from \cite[Proposition 6.1]{MR2609644} that the central charge (see Definition \ref{def:Gauss sum and central charge}) is preserved under taking local modules.

If we complexify, MNEs for $\sVec_\RR$ or $\sVec_\HH$ become pointed by Proposition \ref{prop:The only possible complexifications}.
The central charge of the complexification is multiplicative under tensor product and is preserved by taking local modules.

\begin{lemma}\label{lem:Central charge for MNEs}
    The central charge induces homomorphisms
    \begin{equation*}
        \begin{aligned}
            c:\Mext(\sVec_\DD)&\to\{\pm1\}\subset \RR^\times\\
            (\sVec_\DD\hookrightarrow\C)\hspace{1mm}&\mapsto\hspace{1mm} c\big(\overline\C\big)
        \end{aligned}\;,
    \end{equation*}
    for both $\DD=\RR$, and $\DD=\HH$.
\end{lemma}

\begin{proof}
    By direct computation,
    \begin{align*}
        c\big(\overline{\C\boxdot\D}\big)&=c\big(\overline{(\C\boxtimes\D)_A^0}\big)\\
        &=c\big((\overline{\C}\boxtimes_\CC\overline{\D})_{\overline A}^0\big)\\
        &=c\big(\overline{\C}\boxtimes_\CC\overline{\D}\big)\\
        &=c\big(\overline{\C}\big)\cdot c\big(\overline{\D}\big)\,.
    \end{align*}
\end{proof}

Now that we have the product operation, we are finally able to state the definition of the Mext groups.

\begin{definition}
    Given a symmetric tensor category $\E$ over $\RR$, a minimal nondegenerate extension of $\E$ is a fully faithful braided embedding $i_{\C}:\E\hookrightarrow \C$.
    The underlying set of $\Mext(\E)$ is the set of equivalence classes of minimal nondegenerate extensions, up to the relation that $(i_\C,\C)\sim(i_\D,\D)$ if there exists a braided equivalence $F:\C\to\D$, such that $F\circ i_\C\cong i_\D$ as braided monoidal functors.
    The group structure is induced on equivalences classes by the operation $\boxdot$.
    The identity is given by the canonical embedding $i_\Z:\E\hookrightarrow\Z(\E)$, and the inverse is given by reversing the braiding.
\end{definition}

So far, none of our arguments have relied on the embedding data.
Unlike $\sVec_\CC$, both $\sVec_\RR$ and $\sVec_\HH$ have nontrivial braided autoequivalences coming from the nontrivial class in $H^2(\ZZ/2\ZZ;\RR^\times)\cong\ZZ/2\ZZ$.

\begin{lemma}\label{lem:Don't worry about the embeddings}
    For $\DD=\RR$ or $\DD=\HH$, every embedding $i:\sVec_\DD\hookrightarrow\Z(\sVec_\DD)$ lies in the same $\Mext$ equivalence class as the canonical embedding $i_\Z:\sVec_\DD\hookrightarrow\Z(\sVec_\DD)$.
\end{lemma}

\begin{proof}
    Note that any autoequivalence $F$ of $\sVec_\HH$ must fix all objects.
    Functoriality implies that $F$ acts by algebra automorphisms on $\End(Y)\cong\HH$.
    By the Skolem-Noether theorem, all such automorphisms are inner, so we can assume that $F$ acts by conjugation by $h\in\HH^\times$.
    This implies that $\mu_Y:=h\cdot\id_Y$ defines a natural isomorphism from $F$ to the identity functor.
    
    Any monoidal structure $J$ on the identity functor must be a natural isomorphism, and this means that on $Y$ it must commute with all quaternions.
    In other words, $J$ can only take on real values.
    Since a generic $F$ is naturally isomorphic to the identity functor by $h$, transport of structure reveals that the possible $J$ on $F$ must also take on real values.
    In this way, we can assume that the underlying functor is the identity, and values for $J$ are determined by cocycles in $H^2(\ZZ/2\ZZ;\RR^\times)\cong\ZZ/2\ZZ$.
    It follows that in our current computation, the argument for $\sVec_\HH$ reduces to the related computation for $\sVec_\RR$.

    Suppose then, that $i:\sVec_\RR\hookrightarrow\Z(\sVec_\RR)$ is a fully faithful braided embedding.
    Since the nontrivial object $f\in\sVec_\RR$ is fermionic, it must map to the unique fermionic object $ab\in\Z(\sVec_\RR)=\TC_\RR$.
    Since $\1$ must also map to $\1$, the underlying functor is unique.
    
    There are two possible tensor structures $J$ that can be put on $i$ and they have representatives $J_{f,f}=\id_\1$ and $J_{f,f}=-\id_\1$.
    The fist corresponds to the canonical inclusion, so assume that $J$ is given by this second formula.
    Define a functor $(T,K)$, where $T$ is the functor that swaps the bosons $a\leftrightarrow b$, and the tensor structure is
    \[K_{a^ib^j,a^kb^\ell}=(-1)^{ik+i\ell+j\ell}\cdot\id_{a^{a+k}b^{j+\ell}}\,.\]
    One gauge for $\TC_\RR$ has the braiding given by the formula
    \[\beta_{a^ib^j,a^kb^\ell}=(-1)^{i\ell}\,.\]
    The braiding coherence square is given by the diagram below.
    \[\begin{tikzcd}[ampersand replacement=\&]
    	{T(a^ib^j)\otimes T(a^kb^\ell)} \&\& {T(a^ib^j\otimes a^kb^\ell)} \\
    	{T(a^kb^\ell)\otimes T(a^ib^j)} \&\& {T(a^kb^\ell\otimes a^ib^j)}
    	\arrow["{(-1)^{ik+i\ell+jl}}", from=1-1, to=1-3]
    	\arrow["{(-1)^{jk}}"', from=1-1, to=2-1]
    	\arrow["{(-1)^{i\ell}}", from=1-3, to=2-3]
    	\arrow["{(-1)^{ik+kj+\ell j}}", from=2-1, to=2-3]
    \end{tikzcd}\]
    This commutes, so the functor $(T,K)$ is braided.
    Now notice that $K_{ab,ab}=(-1)\cdot\id_{\1}$, which matches the tensorator $J_{f,f}=(-1)\cdot\id_\1$.
    It follows that $(T,K)\circ(i,J)=i_\Z$, so the two embeddings represent the same class in $\Mext(\sVec_\RR)$.
\end{proof}

We are now ready to prove the main theorems.

\begin{theorem}\label{thm:Mext(sVec_R) is Klein-four}
    The group of $\mathcal Mext(\sVec_{\RR})$ of minimal modular extensions of $\sVec_{\RR}$ is isomorphic to the Klein-four group $(\ZZ/2\ZZ)^2$.
    The four extensions are $\TC_\RR$, $\TC_\HH^b$, $\TF_\RR$, and $\TF_\HH$, with $\TC_\RR$ playing the role of the identity.
\end{theorem}

\begin{proof}
    Minimality of $\C$ implies that $\FPdim(\C)=4$, and so $\C$ must be contained in the list in Theorem \ref{thm:List of all ndBFCs over R of dim=4}.
    Of this list, the only categories that extend $\sVec_\RR$ are those listed in the statement of the theorem.
    All of these categories are $\ZZ/2\ZZ$-graded, and the objects in the nontrivial components braid nontrivially with $\sVec_\RR$, so these are minimal nondegenerate extensions.

    Consider two of these categories, say $\C$ and $\D$.
    Their product is the category of local modules $\C\boxdot\D=(\C\boxtimes\D)_A^0$.
    Because all of these extensions are $\ZZ/2\ZZ$-graded, the category of all modules can be decomposed using the gradings $\C=\C_0\oplus\C_1$ and $\D=\D_0\oplus\D_1$ to give
    \[(\C\boxtimes\D)_A\simeq\C\boxtimes_{\sVec_{\RR}}\D\simeq\bigoplus_{i,j=0}^1\C_{i}\boxtimes_{\sVec_{\RR}}\D_j\simeq \sVec_{\RR}\oplus\C_1\oplus\D_1\oplus\C_1\boxtimes_{\sVec_{\RR}}\D_1\,.\]
    By minimality, the objects in $\C_1$ and $\D_1$ only contain objects that braid nontrivially with the fermion in $\sVec_{\RR}$, and this implies that these objects are not local modules for $A$.
    The objects that remain already constitute a subcategory with $\FPdim=4$, so the underlying category of the product of extensions is
    \[\C\boxdot\D\,:=\;(\C\boxtimes\D)_A^0\simeq \sVec_{\RR}\oplus\C_1\boxtimes_{\sVec_{\RR}}\D_1\,.\]
    Although this is not a complete description of the braiding, we can still deduce relations in $\Mext(\sVec_{\RR})$ by leveraging our knowledge of the product in $\Pic(\sVec_{\RR})$.

    The nontrivial component of $\TC_\HH^b$ and of $\TF_\HH$ is equivalent to $\Mod(\HH)\boxtimes\sVec_\RR$.
    This is just the inclusion of the Brauer group into $\BrPic(\sVec_\RR)$ that was described in \cite{sanford2026puttingbrauerbrauerpicard}, which happens to land in $\Pic(\Vec_\RR)$.
    Since this module has order 2, the minimal extensions $\TC_\HH^b\boxdot\TC_\HH^b$, $\TC_\HH^b\boxdot\TF_\HH$, and $\TF_\HH\boxdot\TF_\HH$ must have only real simple objects.

    We also know that the $\TF$ categories have central charge $-1$.
    By Lemma \ref{lem:Central charge for MNEs}, it follows that the product of two $\TF$ categories must be a $\TC$ category.
    Combining this with the previous observation, we find that $\TC_\RR$ must be the unit element, and the group is generated by $\TC_\HH^b$ and $\TF_\RR$, with $(\ZZ/2\ZZ)^2$ fusion rules.

    The analysis thus far explains the effect of $\boxdot$ on the underlying categories.
    What about the embeddings?
    Suppose that $i_k:\sVec_\RR\hookrightarrow\C$ for $k=1,2$ are two MNEs with the same underlying category.
    The product $\boxdot$ produces an embedding $i:\sVec_\RR\hookrightarrow\TC_\RR$, because all possible $\C$ have order two.
    By Lemma \ref{lem:Don't worry about the embeddings}, this resulting embedding must correspond to the identity element in $\Mext(\sVec_\RR)$.
    But this means that all embeddings for $\C$ are inverse to each other, so all embeddings are inverse to themselves, and therefore all embeddings for a given category $\C$ are equivalent.
\end{proof}

\begin{theorem}\label{thm:Mext(sVec_H) is Klein-four}
    The group of $\mathcal Mext(\sVec_{\HH})$ of minimal modular extensions of $\sVec_{\HH}$ is isomorphic to the Klein-four group $(\ZZ/2\ZZ)^2$.
    The four extensions are $\TC_\HH^f$, $\TC_{\Tri}$, $\TF_\HH$, and $\TF_{\Tri}$, with $\TC_\HH^f$ playing the role of the identity.
\end{theorem}

\begin{proof}
    The argument is similar to that of Theorem \ref{thm:Mext(sVec_R) is Klein-four}.
    The list of Theorem \ref{thm:List of all ndBFCs over R of dim=4} shows that the only categories that extend $\sVec_\HH$ are $\TC_\HH^f$, $\TC_{\Tri}$, $\TF_\HH$, and $\TF_{\Tri}$.
    The fact that these categories are $\ZZ/2\ZZ$-graded means that we can argue, as before, by looking at the product in $\Pic(\sVec_\HH)$.

    The nontrivial component of the triumvirate categories $\TC_{\Tri}$ and $\TF_{\Tri}$ is equivalent to $\Vec_\CC$.
    This module has order two in $\Pic(\sVec_\HH)$, so the product of any two triumvirates must have two real simples and two quaternionic simples.
    The $\TF$ categories have central charge $-1$, so we obtain $(\ZZ/2\ZZ)^2$ fusion rules, with $\TC_\HH^f$ playing the role of the identity.
    The embedding concerns are resolved in the same way as in the $\sVec_\RR$ case.
\end{proof}

\section{Classification of centers of \texorpdfstring{$\FPdim=$}{dimension} 4}\label{sec:Classification of centers}

Now that we have classified the $\Mext$ groups, we would like to investigate the maps
\begin{equation}
    \begin{aligned}
        U_{\DD}:\Mext(\sVec_{\DD})&\to\ker\big(\Witt(\Vec_\RR)\to\Witt(\sVec_{\DD})\big)\\
        (\sVec_{\DD}\hookrightarrow\B)\hspace{1mm}&\mapsto\hspace{9mm}[\B]
    \end{aligned} \hspace{5mm}\text{for }\DD=\RR,\HH\;.
\end{equation}
The letter $U$ stands for underlying, as is common for maps that forget structure.
The complex version of this map was described in \cite[Proposition 5.14]{MR3022755} and was later generalized by Lan Kong and Wen in \cite[Proposition 5.15]{MR3613518}.
We will give more intuition for these maps $U_\DD$ in Section \ref{sec:Mext groups as homotopy fibers}.
For now, we turn our attention to the kernel of $U_\DD$, which is to say, those MNEs whose underlying category is a Drinfeld center.

\subsection{Computing Drinfeld centers with the help of class functions}

Computation of Drinfeld centers can be a challenge, when working over $\RR$.
For those familiar with fusion categories over $\CC$, the presence of non-split simple objects can cause counterintuitive things to happen, as the following example demonstrates.

\begin{example}\label{eg:Q_+ is a monoidal real form of Sem}
    The category $\Q_+:=\C_{\HH}(1,\chi,+1/2)=\langle\1,Y\rangle_{\oplus}$ is a real form of $\Vec_\CC^\omega(\ZZ/2\ZZ)$.
    This category appears to be similar to $\sVec_\HH$, but its associator differs by a sign.
    It does not admit a braiding, because that would imply that $\overline{\Q_+}\simeq\Sem$, and we already saw in Example \ref{eg:Semion has no braided real forms} that $\Sem$ has no braided real forms.

    Having a braiding is equivalent to admitting a fully faithful monoidal embedding $\Q_+\hookrightarrow\Z(\Q_+)$.
    It turns out that $Y$ does admit a half-braiding $\gamma_i:Y\otimes Y\to Y\otimes Y$ that acts by $i\in\End(Y)$, but there is still no embedding.
    Morphisms in $\Z(\Q_+)$ must commute with the half-braidings, and so $\End(Y,\gamma_i)\cong\CC$ consists of only those quaternions that commute with $i$.
    Since there is no algebra map $\HH\to\CC$, there can be no embedding $\Q_+\hookrightarrow\Z(\Q_+)$, despite the fact that all simple objects admit a half-braiding.
\end{example}

Because of these subtleties, this section aims to introduce a technique that will be helpful for identifying centers more consistently.

As was noted in \cite[Proposition 4.2]{MR4806973}, Galois nontrivial objects are \emph{unbraidable}.
This means that they cannot even be subobjects of objects with a half-braiding.
Despite this, they do have an effect on the possible half-braidings of other objects.

At first pass, we can say the following:
\begin{lemma}[{cf. \cite[the proof of Theorem A.1.1]{MR4934638}}]\label{lem:Centers of CG cats are real forms of C_0}
    For any fusion category $\C$ over $\RR$ that is faithfully Galois-graded, $\C=\C_0\oplus\C_1$, $\Z(\C)$ must be a real form of $\Z(\C_0)$.
\end{lemma}

\begin{proof}
    The complexification of a complex, faithfully Galois-graded fusion category must be indecomposable multifusion.
    Because $\CC\otimes_\RR\CC\cong\CC\oplus\CC$, every simple object $X$ splits into two simple objects $pX$ and $qX$ upon complexification, and we can write $p=p\1$ and $q=q\1$.
    The two summands correspond to the identity and complex conjugation via the rules
    \[\id_p\otimes (\id_X\boxtimes\lambda)=\id_p\otimes (\lambda\id_X\boxtimes1)\hspace{3mm}\text{and}\hspace{3mm}\id_q\otimes (\id_X\boxtimes\lambda)=\id_q\otimes (\overline\lambda\id_X\boxtimes1)\;.\]
    It follows that the matrix summands of $\overline\C$ are
    \[\overline\C\simeq\begin{bmatrix}
        p\overline{\C_0}p & p\overline{\C_1}q \\
        q\overline{\C_1}p & q\overline{\C_0}q
    \end{bmatrix}\,.\]
    The center of any indecomposable multifusion category over $\CC$ is braided equivalent to the center of any of its diagonal summands.
    Clearly $\C_0\simeq p\overline{\C_0}p$ as complex fusion categories, so the center of $\C_0$ must be equivalent to $\Z(\overline\C)\simeq\overline{\Z(\C)}$.
\end{proof}

Crucially, this lemma does not tell us \emph{which} real form the center is.
One tool that can help distinguish between possibilities is an isomorphism described by Shimizu in \cite{MR3631720}.
Let $F:\Z(\C)\to\C$ be the forgetful functor.
It has a right adjoint $R:\C\to\Z(\C)$, which preserves algebra objects, because $F$ is monoidal.
The algebra object $FR(\1)$ is a canonical algebra associated to $\C$, and it satisfies $\FPdim(FR(\1))=\FPdim(\C)$.

The adjunction isomorphism 
\begin{equation}
    \C\big(FR(\1),\1\big)\cong\Z(\C)\big(R(\1),R(\1)\big)
\end{equation}
allows for an algebra structure on the hom space $\C(FR(\1),\1)$ via transport of structure.
This algebra is called the algebra of class functions, and is denoted $CF(\C)$.
Shimizu uses a pivotal structure to construct an algebra embedding $\mathrm{ch}:\Gr(\C)\hookrightarrow\CF(\C)$ from the Grothendieck ring, called the internal character.

\begin{proposition}[{cf. \cite[Corollary 4.3]{MR3631720}}]\label{prop:Grothendieck ring isomorphism}
    Let $\C$ be a pivotal fusion category over $\RR$.
    If $\Omega\C=\RR$, then the internal character $\mathrm{ch}:\Gr_{\RR}(\C)\to\CF(\C)$ is an isomorphism from the real Grothendieck ring.
    If $\Omega\C\cong\mathbb C$, then $\mathrm{ch}:\Gr_{\CC,\text{Gal}}(\C)\rightarrow\CF(\C)$ is an isomorphism from a version of the complex Grothendieck ring, where Galois nontrivial simple objects $X\in\C$ correspond to generators with the property that $[X]\lambda=\overline\lambda[X]$.
\end{proposition}

\begin{proof}
    The underlying object of the algebra $FR(\1)$ is
    \[FR(\1)\cong\bigoplus_{X}X\otimes_{\End(X)}X^*\,.\]
    The relative tensor products come from the fact that $R$ can be constructed using an end, and the decomposition into a direct sum is a consequence of $\C$ being semisimple.

    Since Shimizu proves injectivity of $\mathrm{ch}$ more generally, for finite tensor and not just fusion categories, his arguments also use the end formulation, so they apply to our setting.
    In order to prove surjectivity, he specializes to fusion categories in Example 4.4 \emph{loc. cit.}
    The key point is that $\C(X\otimes_{\End(X)}X^*,\1)$ remains one dimensional in our case, and so the map is an isomorphism.

    The morphism $\mathrm{ch}(X\otimes Y)=\mathrm{ch}(X)\mathrm{ch}(Y)$ factors through a map
    \[X\otimes Y\otimes Y^*\otimes X^*\to\1\,.\]
    Scaling $\mathrm{ch}(Y)$ by a complex number $\lambda$ precomposes this map with $\id_X\otimes \lambda\cdot\id_{Y\otimes Y^*\otimes X^*}$.
    If $X$ is Galois nontrivial, then $\id_X\otimes \lambda\cdot\id_Y=\overline{\lambda}\cdot\id_X\otimes \id_Y$, and this produces the desired Galois action.
\end{proof}

\begin{example}\label{eg:Conjugating Ising}
    The category $\C=\C_{\overline{\CC}}(\ZZ/2\ZZ,\chi)$ is a non-split version of a Tambara-Yamagami category that was introduced in \cite{MR5003359}.
    This category has $\Omega\C\cong\CC$, and the simple object $m$ is Galois nontrivial.
    The fusion rules are
    \begin{gather*}
        a\otimes a\cong\1\,,\\
        a\otimes m\cong m\cong m\otimes a\,,\\
        m\otimes m\cong 1\oplus a\,.
    \end{gather*}
    This category admits a pivotal structure, so Proposition  \ref{prop:Grothendieck ring isomorphism} applies.
    To compute in the complex-Galois Grothendieck ring, we will omit brackets for the generators.
    The equation $a^2=1$ tells us that we can build a projection $p=\frac12(1+a)$.
    Writing $1=p+q$, we find that these are central idempotents, so the algebra splits into the images of $p$ and $q$.
    Since $qm=0$ and $qa=-q$, the image of $q$ is isomorphic to $\CC$.
    
    Since $pa=p$, the image of $p$ is generated by $p$ and $pm=m$.
    The fusion rules then imply that
    \[m^2=(1+a)=2p\,.\]
    Thus the element $(\frac{1}{\sqrt2})\cdot m$ squares to $p$ and conjugates scalars.
    This subalgebra can easily be identified as $M_2(\RR)$.
    
    It now follows from Proposition \ref{prop:Grothendieck ring isomorphism} that $\End(R(\1))\cong M_2(\RR)\oplus\CC$.
    Therefore, $\Z(\C)$ contains at least one real simple, and one complex simple.
    Since $\C$ has a Galois nontrivial object, $\Omega\Z\C=\RR$.
    We know that $R(\1)$ contains the unit, because it's an algebra, so the $M_2(\RR)$ summand means that this real simple that we've found is the unit object in $\Z(\C)$.
\end{example}

The drawback of this technique is that it tells us very little.
It describes some simple objects, but not all of them.
It also says nothing about the braiding.
The power lies in the fact that it fills a niche that was unoccupied by our other techniques: it can always be used (assuming $\C$ is pivotal), and it tells us for certain that a given type of simple object exists.

\subsection{Every center of \texorpdfstring{$\FPdim=$}{dimension} 4}

Our minimal nondegenerate extensions all have $\FPdim=4$, and in order to understand $\ker(U_\DD)$, we will need to know which ones are centers.
Our strategy is brute force, case-by-case analysis.
Since 4 is a relatively small number\footnote{...as many important mathematicians have observed.}, case analysis is tractable.

\begin{proposition}[{cf. \cite[Theorem A.1.1]{MR4934638}}]\label{prop:List of all C with dim(ZC)=4}
    Every fusion category $C$ over $\RR$, with the property that $\Omega\Z\C=\RR$ and $\FPdim(\Z(\C))=4$ must be one of the following categories:
    \begin{enumerate}[(I),leftmargin=50mm, labelsep=5mm, itemsep=1mm]
        \item $\Vec_\RR(\ZZ/2\ZZ)$
        \item $\C_\HH(1,\chi,-1/2)=:\Q_-$
        \item $\Vec_\RR^\omega(\ZZ/2\ZZ)$
        \item $\C_\HH(1,\chi,+1/2)=:\Q_+$
        \item $\Vec_\CC\big((\ZZ/2\ZZ)^2_{\text{Gal}}\big)$
        \item $\Vec_\CC^\alpha\big((\ZZ/2\ZZ)^2_{\text{Gal}}\big)$
        \item $\Vec_\CC\big((\ZZ/4\ZZ)_{\text{Gal}}\big)$
        \item $\C_{\overline\CC}(\ZZ/2\ZZ,\chi)$
        \item $\Vec_\CC^\omega\big((\ZZ/2\ZZ)^2_{\text{Gal}}\big)$
        \item $\Vec_\CC^{\alpha\omega}\big((\ZZ/2\ZZ)^2_{\text{Gal}}\big)$.
    \end{enumerate}
    Here, $\Vec_{\KK}^\varphi(G)$ denotes $G$-graded vector spaces over $\KK$, with associator given by the cocycle $\varphi\in Z^3(G;\CC^\times)$.
    The subscript Gal indicates that there is a Galois nontrivial simple object, and that $\varphi$, if present, has twisted $\CC^\times$ coefficients.
    The absence of $\varphi$ indicates the trivial cocycle.
    The categories in items (II), (IV), and (VIII) are non-split Tambara-Yamagami categories, whose classification and notation can be found in \cite{MR5003359}.
\end{proposition}
Exhaustiveness of this list was proven in \cite{MR4934638}, but we include a proof here for completeness.

\begin{proof}
    Let $\C$ be one such category, and suppose at first that $\Omega\C=\RR$.
    It follows that $\FPdim(\C)=2$, so Proposition \ref{prop:Each simple contributes at least 1 to FPdim} implies that each of these categories contains exactly two simple objects, one of which is the unit $\1$.
    The other simple, say $X$, must be self-dual, and must satisfy
    \[1\;=\;\frac{\FPdim(X)^2}{\dim_{\RR}\End(X)}\,.\]
    By Example \ref{eg:Complex simples contribute more}, this object $X$ cannot be complex.
    Since $X$ must either be real or quaternionic, $\FPdim(X)$ must be either $1$ or $2$, respectively.
    If $\FPdim(X)=1$, then $X$ is invertible, and different choices of associator give categories (I) and (III).
    If $\FPdim(X)=2$, then $X$ is quasi-invertible, and different choices of associator give categories (II) and (IV).

    Now suppose that $\Omega\C=\CC$.
    Since $\Omega\Z\C=\RR$, $\C$ must have Galois nontrivial objects, and it follows from \cite[Theorem 4.9]{MR4806973} that $\FPdim(\C)=4$.
    Galois nontriviality induces a $\ZZ/2\ZZ$-grading $\C\simeq\C_0\oplus\C_1$, and this implies that $\FPdim(\C_i)=2$ for each component.
    Since all simples in $\C_0$ are Galois trivial, $\C_0$ must be fusion over $\CC$.
    There are only two possibilities: $\C_0=\Vec_{\CC}(\ZZ/2\ZZ)$, and $\C_0=\Vec_{\CC}^\omega(\ZZ/2\ZZ)$.
    The Galois-grading implies that $\C_1$ must be an invertible bimodule for $\C_0$.
    
    When $\C_0=\Vec_{\CC}(\ZZ/2\ZZ)$, there are two possible bimodules: $\C_1\simeq\C_0$, and $\C_1\simeq\Vec_{\CC}$.
    When $\C_1\simeq\C_0$, the category will be graded by a group of order 4, and this produces categories (V), (VI), and (VII).
    
    Note that twisted coefficients give $H^3((\ZZ/2\ZZ)^2_\text{Gal};\CC^\times_\sim)=(\ZZ/2\ZZ)^2$.
    The four cocycles in $H^3((\ZZ/2\ZZ)^2_\text{Gal};\CC^\times_\sim)$ correspond to categories (V), (VI), (IX), and (X).
    The cocycle $\alpha$ becomes trivial when restricted to $\C_0$, and $\omega$ restricts to be nontrivial.

    When $\C_0=\Vec_{\CC}(\ZZ/2\ZZ)$ and $\C_1\simeq\Vec_\CC$, the category must have Tambara-Yamagami fusion rules, and the classification of \cite[Theorem 7.1]{MR5003359} shows that (VIII) is the only possibility.

    In the case where $\C_0=\Vec_{\CC}^\omega(\ZZ/2\ZZ)$, the only invertible bimodule is the trivial bimodule, and so $\C_1\simeq\C_0$.
    The group $\ZZ/4\ZZ$ has no cocycles that restrict to become $\omega$.
    In fact, conjugating coefficients imply that $H^3(\ZZ/4\ZZ;\CC^\times_\sim)=1$.
    Thus, assuming nontrivial associator, the only possibility is a $(\ZZ/2\ZZ)^2$-grading, and these correspond to categories (IX) and (X) that we already found.
\end{proof}

We will compute the centers of these categories in batches.
The first batch will be those with real unit, corresponding to categories (I)-(IV).
\begin{lemma}\label{lem:The centers of (I)-(IV)}\leavevmode
    \begin{enumerate}
        \item The center of (I) is $\TC_\RR$.
        \item The center of (II) is $\TC_\HH^f$.
        \item The center of (III) and (IV) is $\DS_\RR$.
    \end{enumerate}
\end{lemma}

\begin{proof}
    When the category is pointed, and all the associators are trivial as in case (I), half-braidings are determined by characters.
    The trivial and sign character for $\ZZ/2\ZZ$ can equally be equipped to both invertible objects, so all objects in the center are real, and this is $\TC_\RR$.

    The category (II) is the underlying category of $\sVec_\HH$.
    Theorem \ref{thm:Mext(sVec_H) is Klein-four} already formally identified the unit element of $\Mext(\sVec_\HH)$ as $\TC_\HH^f$, but the unit element of the Mext group is always the Drinfeld center, so $\Z(\Q_-)=\Z(\sVec_\HH)=\TC_\HH^f$.

    The categories (III) and (IV) are the two real forms of $\Vec_{\CC}^\omega(\ZZ/2\ZZ)$, whose center is $\DS$.
    Proposition \ref{prop:DS has a unique real form} implies that the center of each of these categories must be $\DS_\RR$.
\end{proof}

For the second batch, corresponding to categories (V)-(X), we can use Lemma \ref{lem:Centers of CG cats are real forms of C_0} to narrow the possible centers to the real forms of their Galois trivial subcategories.

\begin{lemma}\label{lem:Centers of (V)-(VIII)}\leavevmode
    \begin{enumerate}
        \item The center of (V) is $\TC_\RR$.
        \item The center of (VI) and of (VII) is $\TC_\HH^f$.
        \item The center of (VIII) is $\TC_{\Tri}$.
        \item The center of (IX) and of (X) is $\DS_\RR$.
    \end{enumerate}
\end{lemma}

\begin{proof}
    Lemma \ref{lem:Centers of CG cats are real forms of C_0} tells us that centers of (V)-(VIII) will be real forms of $\TC$, and that the centers of (IX) and (X) must be $\DS_\RR$, by Proposition \ref{prop:DS has a unique real form}.

    Using proposition \ref{prop:Grothendieck ring isomorphism}, the complex-Galois Grothendieck ring of (V) is easily seen to be $M_2(\RR)\oplus M_2(\RR)$.
    Since $FR(\1)\cong 4\cdot\1$, it follows that the unit object in (V) admits two nonisomorphic half-braidings.
    These half-braidings correspond to characters for $\ZZ/2\ZZ$, determined by how the object braids across the Galois trivial simple object $a$.
    The object $a$ can be equipped with such characters as well, and this establishes the existence of four distinct real simples, so the center must be $\TC_\RR$.

    Let us write $a$ for the generator of $\ZZ/4\ZZ$ in case (VII).
    The Grothendieck ring contains a central involution $a^2$, so the projections $p=\tfrac12(1+a^2)$ and $q=\tfrac12(1-a^2)$ decompose the algebra into two blocks.
    On the $p$ block, $(pa)^2=p$ and $pa$ conjugates scalars, so this block is isomorphic to $M_2(\RR)$.
    On the $q$ block, $(qa)^2=-q$ and $qa$ conjugates scalars, so this block is isomorphic to $\HH$.
    Since $FR(\1)\cong 4\cdot\1$, we have found real and quaternionic half-braidings that live over $\1$.
    The quaternionic half-braiding can be realized as a $90^\circ$ rotation matrix acting on $2\cdot\1$ whenever the object braids over $a$.
    These half-braidings can be ported to live over $a^2$ as well.
    The quaternionic half-braiding on $a^2$ is easily seen to be fermionic, from the description above, and so the center must be $\TC_\HH^f$.
    It was shown in \cite[Theorem V.6]{wen2026categoricaltimereversalsymmetries} that (VI) and (VII) are Morita equivalent, and so the center of (VI) is also $\TC_\HH^f$.

    Finally in case (VIII), we have already seen in Example \ref{eg:Conjugating Ising} that $\End(R(\1))\cong M_2(\RR)\oplus\CC$.
    By Proposition \ref{prop:Classification of real forms of TC and 3F}, the presence of a complex object implies that the center must be $\TC_{\Tri}$.
\end{proof}

\begin{theorem}\label{thm:All the centers of dim=4}
    Suppose $\C$ is a fusion category over $\RR$, with $\Omega\Z\C:=\End(\1_{\Z(\C)})\cong\RR$ and $\FPdim(\Z(\C))=4$.
    Then the pair $(C,\Z(\C))$ corresponds to exactly one of the rows in the table below.
    \begin{equation}\label{eqn:Table of centers}
        \begin{array}{c|c|c}
           \text{label} & \C & \Z(\C) \\\hline\\[-8pt]
          (I) & \Vec_\RR(\ZZ/2\ZZ) & \TC_\RR\\[2pt]
          (II) & \C_\HH(1,\chi,-1/2) & \TC_\HH^f\\[2pt]
          (III) & \Vec_\RR^\omega(\ZZ/2\ZZ) & \DS_\RR\\[2pt]
          (IV) & \C_\HH(1,\chi,+1/2) & \DS_\RR\\[2pt]
          (V) & \Vec_\CC\big((\ZZ/2\ZZ)^2_{\text{Gal}}\big) & \TC_\RR\\[2pt]
          (VI) & \Vec_\CC^\alpha\big((\ZZ/2\ZZ)^2_{\text{Gal}}\big) & \TC_\HH^f\\[2pt]
          (VII) & \Vec_\CC\big((\ZZ/4\ZZ)_{\text{Gal}}\big) & \TC_\HH^f\\[2pt]
          (VIII) & \C_{\overline{\CC}}(\ZZ/2\ZZ,\chi) & \TC_{\Tri}\\[2pt]
          (IX) & \Vec_\CC^\omega\big((\ZZ/2\ZZ)^2_{\text{Gal}}\big) & \DS_\RR\\[2pt]
          (X) & \Vec_\CC^{\alpha\omega}\big((\ZZ/2\ZZ)^2_{\text{Gal}}\big) & \DS_\RR\\
        \end{array}
    \end{equation}
\end{theorem}

\begin{proof}
    Combine Proposition \ref{prop:List of all C with dim(ZC)=4} with Lemmas \ref{lem:The centers of (I)-(IV)} and \ref{lem:Centers of (V)-(VIII)}.
\end{proof}

The category $\TC_\HH^b$ doesn't appear in Table (\ref{eqn:Table of centers}), but its complexification is $\TC$, which is a center.

\begin{corollary}\label{cor:Bosonic TC_H is the true class}
    The category $\TC_\HH^b$\footnote{not $\TC_{\Tri}$ as previously claimed in \cite[Appendix A]{MR4934638}} is not a Drinfeld center, and thus represents the unique nontrivial class in
    \[\ker\big(\Witt(\Vec_\RR)\to\Witt(\Vec_\CC)\big)\cong\ZZ/2\ZZ\,.\]
\end{corollary}

\begin{proof}
    Being nonzero in this kernel is precisely the statement that $\TC_\HH^b$ is not a center, but becomes one upon complexification.
    It was already proven in \cite{MR4934638} that the kernel of the complexification functor is $H^4(\RR;\mathbb G_m)\cong\ZZ/2\ZZ$, so this class must be unique.
\end{proof}

\begin{remark}
    Nontriviality of this class can be interpreted as an anomaly for time-reversal symmetry fractionalization (see \cite[Example VI.2]{wen2026categoricaltimereversalsymmetries}).
    We expect that the techniques of \cite{MR4072236} can be leveraged to show that this nontrivial class can be detected by evaluating the corresponding Crane-Yetter TQFT on the manifold $\mathbb{RP}^4$.
\end{remark}

\begin{corollary}\label{cor:No kernel for Mext(sVec_R)}
    The map $U_\RR:\C\mapsto[\C]$ from $\Mext(\sVec_\RR)$ into $\Witt(\Vec_\RR)$ yields an identification
    \[\Mext(\sVec_\RR)\cong\ker\big(\Witt(\Vec_\RR)\to\Witt(\sVec_\RR)\big)\,.\]
\end{corollary}

\begin{proof}
    The only category from Theorem \ref{thm:Mext(sVec_R) is Klein-four} that appears in Table (\ref{eqn:Table of centers}) is $\TC_\RR$, which is the unit.
    Thus $\ker(U_\RR)=1$, so $U_\RR$ in injective.
    The fact that $U_\RR$ surjects onto the kernel follows from the same argument given in \cite[Proposition 5.15]{MR3613518}.
\end{proof}

A similar comparison of Table (\ref{eqn:Table of centers}) with Theorem \ref{thm:Mext(sVec_H) is Klein-four} shows that some of these MNEs \emph{are} Witt trivial over $\RR$.

\begin{corollary}\label{cor:Some kernel for Mext(sVec_H)}
    The map $U_\HH:\C\mapsto[\C]$ from $\Mext(\sVec_\HH)$ to $\Witt(\Vec_\RR)$ has kernel $\ZZ/2\ZZ$; generated by $\TC_{\Tri}$.
    This implies that $\TF_\HH$ and $\TF_{\Tri}$ are Witt equivalent.
\end{corollary}

\begin{proof}
    The argument is similar to Corollary \ref{cor:No kernel for Mext(sVec_R)}.
    Witt triviality of $\TC_{\Tri}$ implies that $\TF_\HH\simeq\TF_{\Tri}\boxdot\TC_{\Tri}$ is Witt equivalent to $\TF_{\Tri}$.
\end{proof}

The final interesting thing about this table is that $\TC_{\Tri}=\Z(\C_{\overline\CC}(\ZZ/2\ZZ,\chi))$.
For some context, the $\E$ center of a fusion category $\C$ over $\E$ is the centralizer of $\E$ inside of $\Z(\C)$, and is denoted $\Z(\C,\E)$.
An $\E$-Witt equivalence, from $\mathcal A$ to $\B$ is a fusion category $\C$ over $\E$, equipped with a braided equivalence $\mathcal A\boxtimes_\E\B^{rev}\to\Z(\C,\E)$ as braided categories over $\E$.

\begin{proposition}\label{prop:Nontrivial auto-Witt equivalence for sVec_H}
    The fusion category $\C_{\overline{\CC}}(\ZZ/2\ZZ,\chi)$ is a non-Morita-trivial $\sVec_\HH$-Witt equivalence from $\sVec_\HH$ to itself.
\end{proposition}

\begin{proof}

    Let $\E=\sVec_\HH$, and $\C=\C_{\overline{\CC}}(\ZZ/2\ZZ,\chi)$.
    The fact that $\Z(\C)$ is a minimal nondegenerate extension of $\sVec_\HH$ is equivalent to the statement that there is a braided equivalence $\sVec_\HH\simeq\Z(\C,\sVec_\HH)$, over $\sVec_\HH$.
    Since $\sVec_\HH$ is symmetric, we can write this as
    \[\sVec_\HH\boxtimes_{\sVec_\HH}\sVec_\HH^{rev}\simeq\Z(\C,\sVec_\HH),\]
    which is just the statement that $\C$ (together with the braided equivalence) is an $\sVec_\HH$-Witt equivalence from $\sVec_\HH$ to itself.

    Nontriviality means that $\C$ is not Morita equivalent to $\sVec_\HH$.
    This is immediate, because they have inequivalent Drinfeld centers.
\end{proof}


\section{Interpretation of \texorpdfstring{$\Mext$}{Mext} groups as homotopy fibers}\label{sec:Mext groups as homotopy fibers}

This section is a speculative outlook that should help organize all the ideas proved earlier in the paper.
I make no claims of originality or rigor in this section; my aim here is only to explain what is expected to be true.

The Witt groups of $\E$ can be interpreted as $\pi_0$ of a higher groupoid $\underline{\Witt}(\E)$.
The objects of $\underline{\Witt}(\E)$ are braided categories over $\E$.
The 1-morphisms are the $\E$-Witt equivalences.
The 2-morphisms $\C\to\D$ are central $\C$-$\D$ bimodules.
The 3-morphisms are functors of such bimodules, and the 4-morphisms are natural transformations of such functors.

The homotopy groups of $\underline{\Witt}(\Vec_\CC)$ are
\begin{itemize}[~]
    \item $\pi_4\underline{\Witt}(\Vec_\CC)\cong\CC^\times$,
    \item $\pi_3\underline{\Witt}(\Vec_\CC)=0$,
    \item $\pi_2\underline{\Witt}(\Vec_\CC)=0$,
    \item $\pi_1\underline{\Witt}(\Vec_\CC)=0$,
    \item $\pi_0\underline{\Witt}(\Vec_\CC)\cong\Witt(\Vec_\CC)$,
\end{itemize}
and the homotopy groups of $\underline{\Witt}(\sVec_\CC)$ are
\begin{itemize}[~]
    \item $\pi_4\underline{\Witt}(s\Vec_\CC)\cong\CC^\times$,
    \item $\pi_3\underline{\Witt}(\sVec_\CC)=\ZZ/2\ZZ$,
    \item $\pi_2\underline{\Witt}(\sVec_\CC)=\ZZ/2\ZZ$,
    \item $\pi_1\underline{\Witt}(\sVec_\CC)=0$,
    \item $\pi_0\underline{\Witt}(\Vec_\CC)\cong\Witt(\sVec_\CC)$.
\end{itemize}
The the nontrivial element in $\pi_3\underline{\Witt}(\sVec_\CC)$ corresponds to the fermion in $\sVec_\CC$, and the copy of $\ZZ/2\ZZ$ in $\pi_2\underline{\Witt}(\sVec_\CC)$ corresponds the Brauer-Wall group, with the nontrivial element corresponding to the Clifford algebra $Cl(1)$.

For a general $\E$ over a field $\KK$, we expect to find
\begin{itemize}[~]
    \item $\pi_4\underline{\Witt}(\E)\cong\KK^\times$,
    \item $\pi_3\underline{\Witt}(\E)=\text{invertible objects in }\E$,
    \item $\pi_2\underline{\Witt}(\E)=\text{Morita-invertible algebras in }\E$,
    \item $\pi_1\underline{\Witt}(\E)=\E\text{-Witt equivalences}$,
    \item $\pi_0\underline{\Witt}(\E)\cong\Witt(\E)$.
\end{itemize}

The exact construction of these higher groupoids would require a good definition of a symmetric monoidal 4-category, but we can think of them as topological spaces through a classifying space construction, similar to those described in \cite[Section 7.1]{MR2677836}, and analogous to classifying spaces of groups.
Authors more inclined to $\infty$-category theory might instantiate them as $(\infty,0)$-categories.
For example, several constructions of the $E_2$-Morita 4-category of braided fusion categories have been described in the literature \cite{MR4228258,MR4302495,MR3590516,MR5047948,décoppet2024classificationfusion2categories,MR3650080,Solr-DE-604.BV050626861}, and $\underline{\Witt}(\E)$ corresponds to the subgroupoid consisting of invertible morphisms at all levels.

However these objects are to be modeled, the $\Mext$ groups are defined in such a way that they should also compile into a higher groupoid $\underline{\Mext}(\E)$ and we expect this groupoid to have the following property:

\begin{conjecture}[folklore]\label{conj:The only conjecture in the paper}
    For a symmetric fusion category $\E$ over a field $\KK$, there is a homotopy fiber sequence
    \[\underline{\Mext}(\E)\hookrightarrow\underline{\Witt}(\Vec_\KK)\twoheadrightarrow\underline{\Witt}(\E)\,.\]
\end{conjecture}
I learned of this conjecture from Thibault Décoppet, and I expect that Lan, Kong, and Wen understood this when they wrote \cite{MR3613518}.
It was certainly known by Johnson-Freyd and Reutter, as hints of it can be found in \cite[Section 4]{MR4654609}.

By setting $\KK=\RR$ and $\E=\sVec_\RR$, and looking at the long exact sequence of homotopy groups associated to the fiber sequence of Conjecture \ref{conj:The only conjecture in the paper}, we see the following:
\[
    \begin{tikzpicture}[x=3cm,y=1.2cm]
        \node (U) at (0.045,4) {$\pi_4\underline{\Mext}(\sVec_\RR)$};
        \node (V) at (1,4) {$\RR^\times$};
        \node (W) at (2,4) {$\RR^\times$};
        \node (X) at (0.045,3) {$\pi_3\underline{\Mext}(\sVec_\RR)$};
        \node (Y) at (1,3) {$0$};
        \node (Z) at (2,3) {$\ZZ/2\ZZ$};
        \node (A) at (0.045,2) {$\pi_2\underline{\Mext}(\sVec_\RR)$};
        \node (B) at (1,2) {$\ZZ/2\ZZ$};
        \node (C) at (2,2) {$\ZZ/8\ZZ$};
        \node (D) at (0.045,1) {$\pi_1\underline{\Mext}(\sVec_\RR)$};
        \node (E) at (1,1) {$0$};
        \node (F) at (2,1) {$\pi_1\underline{\Witt}(\sVec_\RR)$};
        \node (G) at (0,0) {$\Mext(\sVec_\RR)$};
        \node (H) at (1,0) {$\Witt(\Vec_\RR)$};
        \node (I) at (2,0) {$\Witt(\sVec_\RR)$.};
        \draw[->, rounded corners]
        (U) edge (V)
        (V) edge (W)
        (W) -- ++(0.6,0) -- ++(0,-.5) -- ++(-3.1,0) -- ++(0,-.5) -- ++(.1,0) edge (X)
        (X) edge (Y)
        (Y) edge (Z)
        (Z) -- ++(0.6,0) -- ++(0,-.5) -- ++(-3.1,0) -- ++(0,-.5) -- ++(.1,0) edge (A)
        (A) edge (B)
        (B) edge (C)
        (C) -- ++(0.6,0) -- ++(0,-.5) -- ++(-3.1,0) -- ++(0,-.5) -- ++(.1,0) edge (D)
        (D) edge (E)
        (E) edge  (F)
        (F) -- ++(0.6,0) -- ++(0,-.4) -- ++(-3.1,0) -- ++(0,-.6) -- ++(.1,0) edge (G)
        (G) -- node[above]{$U_\RR$} (H)
        (H) edge (I);
    \end{tikzpicture}
\]
Here, we find $\ZZ/8\ZZ$ real Bott periodicity in the form of the Brauer-Wall group for $\RR$, which we can think of as invertible algebras in $\sVec_\RR$.
The map from the Brauer group $\ZZ/2\ZZ$ of $\RR$ is the inclusion of the quaternions, and the map from $\RR^\times\to\RR^\times$ is the identity.
The fact that $\pi_1\underline{\Witt}(\Vec_\RR)=0$ comes from the computation in \cite{sanford2024invertiblefusioncategories} that the group of Morita-invertible fusion categories over $\KK$ in are in bijection with $H^3(\KK;\mathbb G_m)$.

From this we expect the following homotopy groups
\begin{itemize}[~]
    \item $\pi_4\underline{\Mext}(s\Vec_\RR)=0$,
    \item $\pi_3\underline{\Mext}(\sVec_\RR)=0$,
    \item $\pi_2\underline{\Mext}(\sVec_\RR)=\ZZ/2\ZZ$,
    \item $\pi_1\underline{\Mext}(\sVec_\RR)=\ZZ/4\ZZ$,
    \item $\pi_0\underline{\Mext}(\Vec_\RR)\cong\Mext(\sVec_\RR)\mathop{\cong}\limits^{Thm \ref{thm:Mext(sVec_R) is Klein-four}}(\ZZ/2\ZZ)^2$.
\end{itemize}
If Conjecture \ref{conj:The only conjecture in the paper} is true, we see that our computation that $\ker(U_\RR)=0$ implies that $\pi_1\underline{\Witt}(\sVec_\RR)$ should be zero.

Unsurprisingly, given the fact that $\sVec_\HH$ remains relatively unexplored, we have a lot of unknowns in the version with $\E=\sVec_\HH$.
\[
    \begin{tikzpicture}[x=3cm,y=1.2cm]
        \node (U) at (0.046,4) {$\pi_4\underline{\Mext}(\sVec_\HH)$};
        \node (V) at (1,4) {$\RR^\times$};
        \node (W) at (2,4) {$\RR^\times$};
        \node (X) at (0.046,3) {$\pi_3\underline{\Mext}(\sVec_\HH)$};
        \node (Y) at (1,3) {$0$};
        \node (Z) at (2,3) {$\pi_3\underline{\Witt}(\sVec_\HH)$};
        \node (A) at (0.046,2) {$\pi_2\underline{\Mext}(\sVec_\HH)$};
        \node (B) at (1,2) {$\ZZ/2\ZZ$};
        \node (C) at (2,2) {$\pi_2\underline{\Witt}(\sVec_\HH)$};
        \node (D) at (0.046,1) {$\pi_1\underline{\Mext}(\sVec_\HH)$};
        \node (E) at (1,1) {$0$};
        \node (F) at (2,1) {$\pi_1\underline{\Witt}(\sVec_\HH)$};
        \node (G) at (0,0) {$\Mext(\sVec_\HH)$};
        \node (H) at (1,0) {$\Witt(\Vec_\RR)$};
        \node (I) at (2,0) {$\Witt(\sVec_\HH)$.};
        \draw[->, rounded corners]
        (U) edge (V)
        (V) edge (W)
        (W) -- ++(0.6,0) -- ++(0,-.5) -- ++(-3.1,0) -- ++(0,-.5) -- ++(.1,0) edge (X)
        (X) edge (Y)
        (Y) edge (Z)
        (Z) -- ++(0.6,0) -- ++(0,-.5) -- ++(-3.1,0) -- ++(0,-.5) -- ++(.1,0) edge (A)
        (A) edge (B)
        (B) edge (C)
        (C) -- ++(0.6,0) -- ++(0,-.5) -- ++(-3.1,0) -- ++(0,-.5) -- ++(.1,0) edge (D)
        (D) edge (E)
        (E) edge  (F)
        (F) -- ++(0.6,0) -- ++(0,-.4) -- ++(-3.1,0) -- ++(0,-.6) -- ++(.1,0) edge (G)
        (G) -- node[above]{$U_\HH$} (H)
        (H) edge (I);
    \end{tikzpicture}
\]
Even though $Y\in\sVec_\HH$ is quasi-invertible, $\pi_3\underline{\Witt}(\sVec_\HH)$ should be zero, because invertibility should be up to isomorphism and not Morita equivalence at this level.
The invertible algebra objects in $\sVec_\HH$ have yet to be determined, but this may be a tractable computation.
If Conjecture \ref{conj:The only conjecture in the paper} is true, then this would imply that $\ker(U_\HH)$ ($=\ZZ/2\ZZ$ by Corollary \ref{cor:Some kernel for Mext(sVec_H)}) is isomorphic to $\pi_1\underline{\Witt}(\sVec_\HH)$.
This is exactly the phenomena we observed in Proposition \ref{prop:Nontrivial auto-Witt equivalence for sVec_H}.

We conclude with some open questions about these structures.

\begin{question}
    Analogous to the Brauer and Brauer-Wall groups, what is the group of Morita invertible algebra objects in $\sVec_\HH$? (this should correspond to $\pi_2\underline{\Witt}(\sVec_\HH)$.)
\end{question}

Homotopy fibers in algebraic topology can be described via a path space construction.
This construction makes sense in the $\infty$-categorical framework, and since the Witt groupoids can be modeled as $(\infty,0)$-categories, it can be implemented in this context as well.
Specific versions of this homotopy fiber construction for Brauer-Picard groupoids have appeared in \cite[Section 5]{MR3354332} and \cite[Section 7]{sanford2026puttingbrauerbrauerpicard}.

\begin{question}
    Can we compute the higher homotopy groups of the homotopy fiber directly by using the higher categorical interpretation of homotopy fibers? (conjecturally, these are the $\pi_*\underline{\Mext}(\sVec_\DD)$ groups.)
\end{question}

In the fiber sequence described in \cite[Theorem 7.3]{sanford2026puttingbrauerbrauerpicard}, quasi-invertible objects like $Y\in\sVec_\HH$ do not count as true invertible objects, but they do cause copies of $\ZZ/2\ZZ$ to appear elsewhere.

\begin{question}
    What role does quasi-invertibility play in this sequence?
    Is the presence of $Y\in\sVec_\HH$ reflected in $\pi_2\underline{\Mext}(\sVec_\HH)$? 
\end{question}

\emergencystretch 2em 
\printbibliography

\end{document}